\documentclass[11pt]{article}
\usepackage{amsfonts,amssymb,amsmath,comment}
\usepackage{graphicx,graphics,color}

\usepackage[T1]{fontenc}

\makeatletter
\@addtoreset{equation}{section}
\makeatother

\newtheorem{proposition}{Proposition}[section]
\newtheorem{theorem}[proposition]{Theorem}

\newtheorem{lemma}[proposition]{Lemma}
\newtheorem{corollary}[proposition]{Corollary}

\newtheorem{conjecture}[proposition]{Conjecture}

\newcommand{\halmos}{\rule{1ex}{1.4ex}}
\newenvironment{proof}{\noindent {\em Proof}.\,\,}
{\hspace*{\fill}$\halmos$\medskip}

\def\d{\mathrm{d}}
\def\e{\mathrm{e}}
\def\N{\mathbb{N}}
\def\R{\mathbb{R}}
\def\Z{\mathbb{Z}}
\def\dd{\mathrm{d}}
\def\ee{\mathrm{e}}
\def\E{\mathbb{E}}
\def\P{\mathbb{P}}
\def\cL{\mathcal{L}}
\newcommand{\expec}{\mathbb{E}}
\newcommand{\loweq}{\underline{\nu}}
\newcommand{\upeq}{\overline{\nu}}
\newcommand{\Expect}{\mathbf{E}}
\newcommand{\Prob}{\mathbf{P}}
\newcommand{\whp}{\mathrm{whp}}

\def\Om{\Omega}
\def\meta{\mathbf{m}}
\def\stab{\mathbf{s}}
\def\crit{\mathbf{c}}

\def\d{{\rm d}}
\def\ee{{\rm e}}

\def\crit{\mathbf{c}}

\newcommand{\PP}{\mathbb{P}}

\newcommand{\dmin}{d_\mathrm{min}}
\newcommand{\dave}{d_\mathrm{ave}}

\newcommand{\cC}{\mathcal{C}}

\newcommand{\cT}{\mathcal{T}}
\def\cD{\mathcal{D}}

\def\cO{\mathcal{O}}
\def\cI{\mathcal{I}}

\newcommand{\CM}{\mathrm{CM}}

\newcommand{\XX}{\mathcal{X}}

\begin{document}

\title{Interacting Particle Systems on Random Graphs}

\author{F.\ Capannoli, F.\ den Hollander
\footnote{Mathematical Institute, Leiden University, Einsteinweg 55, 2333 CC Leiden, The Netherlands.}
}

\date{}

\maketitle

\begin{abstract}
The present overview of interacting particle systems on random graphs collects the notes of a mini-course given by the authors at the Brazilian School of Probability, 5--9 August 2024, in Salvador, Bahia, Brazil. The content is a personal snapshot of an interesting area of research at the interface between probability theory, combinatorics, statistical physics and network science that is developing rapidly.
\end{abstract}



\medskip\noindent
\emph{Acknowledgement:}
The authors are supported by the Netherlands Organisation for Scientific Research (NWO) through Gravitation Grant NETWORKS-024.002.003. FC is also supported by the European Union’s Horizon 2020 research and innovation programme under the Marie Skłodowska-Curie grant agreement no.\ 945045.

\vspace{0.1cm}
\hfill\includegraphics[scale=0.1]{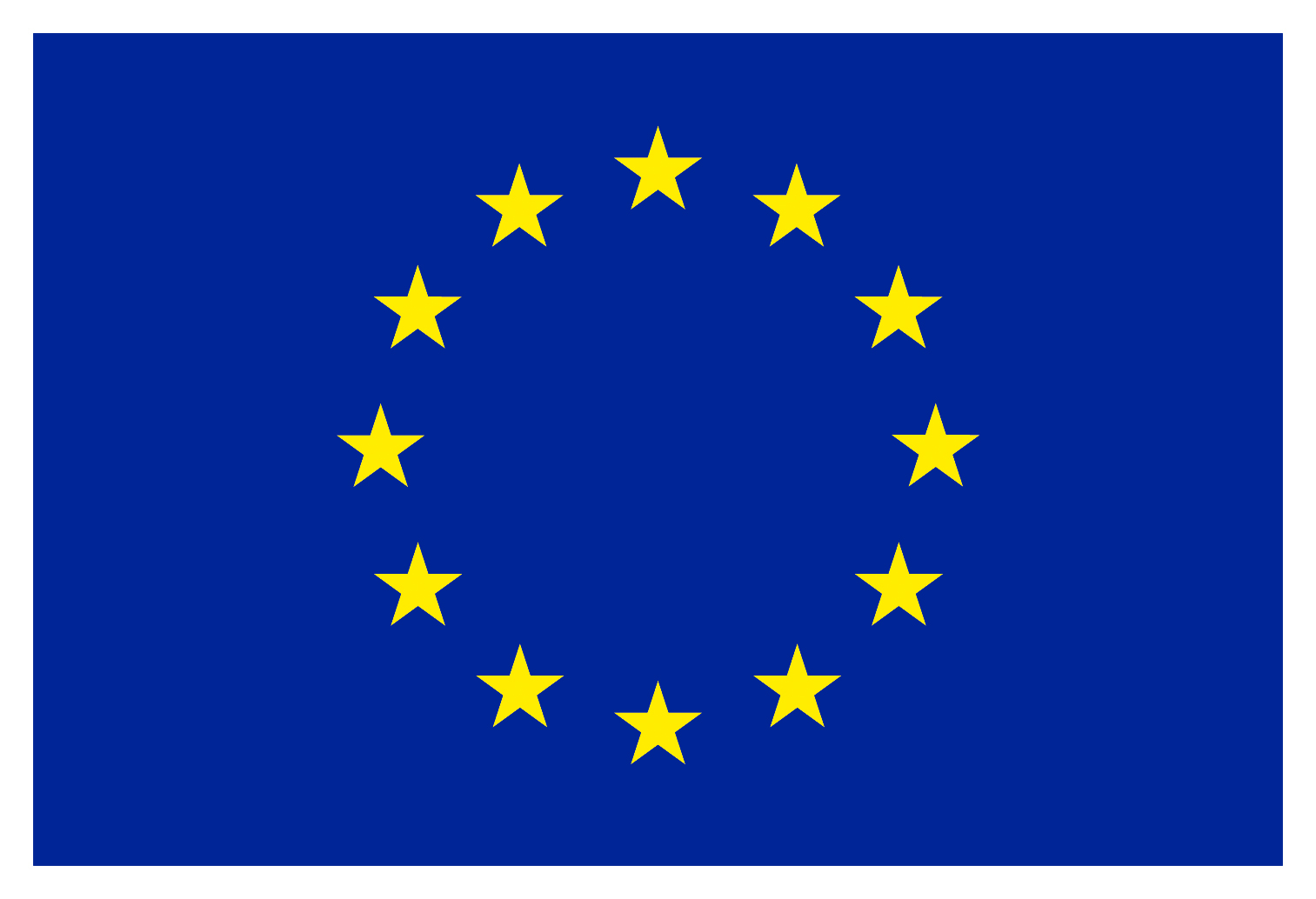}


\newpage

\small
\tableofcontents
\normalsize

\newpage



\addtocounter{section}{-1}

\section{An overview: structure and goals}
\label{s.introduct}

In the present overview we focus on \emph{interacting particle systems on random graphs}. This is a relatively new research area in which progress is rapid, problems are challenging, and new panoramas unfold. It lies at the interface between \emph{probability theory}, \emph{combinatorics}, \emph{statistical physics} and \emph{network science}, and as such is captivating. 

The overview is structured as follows: 
\begin{itemize}
\item[$\blacktriangleright$]
Section~\ref{sec:L1}: Lecture 1. Background and motivation for Interacting Particle Systems (IPS) on $\Z^d$, $d \geq 1$. Key questions and core tools. Phase transitions.
\item[$\blacktriangleright$]
Section~\ref{sec:L2}: Lecture 2. Stochastic Ising Model (SIM) on random graphs.
\item[$\blacktriangleright$]
Section~\ref{sec:L3}: Lecture 3. Voter Model (VM) on random graphs.
\item[$\blacktriangleright$]
Section~\ref{sec:L4}: Lecture 4. Contact Process (CP) on random graphs.
\end{itemize}

\noindent
In Lectures 2--4, IPS on \emph{four classes of random graphs} will be considered:
\begin{itemize}
\item[$\bullet$]
Homogeneous Erd\H{o}s-R\'enyi Random Graph (HER).
\item[$\bullet$]
Inhomogeneous Erd\H{o}s-R\'enyi Random Graph (IER).
\item[$\bullet$]
Configuration Model (CM).
\item[$\bullet$]
Preferential Attachment Model (PAM).
\end{itemize}

\noindent
The goal of the overview is to:
\begin{itemize}
\item[--]
Sketch what is known and not known about IPS on random graphs.
\item[--]
Highlight the role of sparse versus dense graphs.
\item[--]
Exhibit the relevant time scales for critical phenomena and identify how these depend on the size of the graph.   
\item[--]
List some open problems and indicate some lines of future research. 
\end{itemize}

\noindent
Standard references for IPS on $\Z^d$ are Liggett \cite{L85,L99}. Standard references for random graphs are van der Hofstad \cite{vdH17,vdH24}. In what follows, proofs of theorems are sometimes included but often not. For details the reader is referred to the relevant references that are given.


\section{LECTURE 1: Interacting Particle Systems on $\Z^d$, $d \geq 1$. General properties and three examples}
\label{sec:L1}

As an area of research, IPS started in the 1970s, with pioneers Spitzer, Dobrushin, Harris, Holley, Stroock, Liggett, Griffeath, Durrett. Over the years, IPS has turned out to be a fertile breeding ground for the development of new ideas and techniques in mathematical statistical physics, including \emph{graphical representation}, \emph{coupling}, \emph{duality} and \emph{correlation inequalities}.


\subsection{Spin-flip systems}

We start by defining what an IPS is. We focus on \emph{spin-flip systems}, which constitute a particularly tractable class. Within this class we focus on \emph{three examples}: 

\medskip
\begin{tabular}{ll}
&Stochastic Ising Model (SIM)\\
&Voter Model (VM)\\
&Contact Process (CP)
\end{tabular}

\medskip\noindent
Standard references for IPS on $\Z^d$ are Liggett \cite{L85,L99}. For most of the results to be described in this section, references can be found in these monographs.

\medskip\noindent
{\bf Definition.}
An Interacting Particle System (IPS) is a Markov process $\xi=(\xi_t)_{t\geq0}$ on the state space 
$$
\Omega=\{0,1\}^{\Z^d}, \qquad d \geq 1,
$$
where
$$
\xi_t = \{\xi_t(x)\colon\,x\in\Z^d\}
$$
denotes the configuration at time $t$, with $\xi_t(x)=1$ or $0$ meaning that there is a `particle' or a `hole' at site $x$ at time $t$, respectively. Alternative interpretations are:
$$
\begin{array}{rcl}
1 & = & \textrm{spin-up/democrat/infected}\\
0 & = & \textrm{spin-down/republican/healthy}.
\end{array}
$$
\hfill$\spadesuit$

\noindent
The configuration changes over time, which models how:

\medskip
\begin{tabular}{ll} 
$\bullet$ magnetic atoms flip up and down as a result of noise,\\ 
$\bullet$ two political parties evolve in an election campaign,\\ 
$\bullet$ a virus spreads through a population.
\end{tabular}

\medskip\noindent
The evolution is specified via a set of \emph{local transition rates}
$$
c(x,\eta), \qquad x\in \Z^d,\,\eta\in\Omega,
$$
playing the role of the rate at which the state at site $x$ changes in the configuration $\eta$, i.e.,
$$
\eta\to \eta^x
$$
with $\eta^x$ the configuration obtained from $\eta$ by changing the state at site $x$ (either $0 \to 1$ or $1 \to 0$). Since there are only two possible states at each site, the IPS is called a \emph{spin-flip system}.

If $c(x,\eta)$ depends on $\eta$ only via $\eta(x)$, the value of the spin at site $x$, then $\xi$ consists of independent spin-flips. In general, however, the rate to flip the spin at $x$ depends on the spins located in the neighbourhood of $x$ (possibly even on all spins). This dependence models an \emph{interaction} between the spins at different sites. 

In order for $\xi$ to be well-defined, some restrictions must be placed on the local transition rates: $c(x,\eta)$ must depend only weakly on the states at far away sites (formally, $\eta\mapsto c(x,\eta)$ is continuous in the product topology) and must be not too large (formally, bounded away from infinity in some appropriate sense). See Liggett \cite[Chapter I]{L85} for details.


\subsection{Shift-invariant attractive systems}

A typical assumption is that
$$
c(x,\eta) = c(x+y,\tau_y\eta), \qquad y\in\Z^d,
$$
with $\tau_y$ the shift of space over $y$, i.e., 
$$
(\tau_y\eta)(x) = \eta(x-y), \qquad x\in\Z^d.
$$
This property says that the flip rate at $x$ only depends on the configuration $\eta$ seen relative to $x$, which is natural when the interaction between spins is homogeneous in space. Spin-flip systems with this property are called \emph{shift-invariant}.

Another useful assumption is that the interaction favours spins that are alike, i.e.,
$$
\eta\preceq\eta' \to \left\{\begin{array}{lr}
c(x,\eta) \leq c(x,\eta') &\textrm{if } \eta(x) = \eta'(x)=0,\\
c(x,\eta) \geq c(x,\eta') &\textrm{if } \eta(x) = \eta'(x)=1,
\end{array}\right.
$$
where $\preceq$ denotes the partial order in $\Omega$. This property says that the spin at $x$ flips up faster in $\eta'$ than in $\eta$ when $\eta'$ is everywhere larger than $\eta$, and flips down slower. In other words, the dynamics preserves $\preceq$. Spin-flip systems with this property are called \emph{attractive}.


\subsection{Three examples}

\begin{enumerate}
\item 
SIM: \emph{Stochastic Ising model}.\\[0.2cm]
This model is defined on $\Omega=\{-1,1\}^{\Z^d}$ with rates
$$
c(x,\eta) = \exp\left[-\beta\eta(x)\sum_{y \sim x}\eta(y)\right], \qquad \beta \geq 0,
$$
which means that spins prefer to align with the majority of the neighbouring spins.
\item
VM: \emph{Voter model}.\\[0.2cm]
This model is defined on $\Omega=\{0,1\}^{\Z^d}$ with rates
$$
c(x,\eta) = \frac{1}{2d}\sum_{y\sim x} 1_{\{\eta(y)\not=\eta(x)\}},
$$
which means that sites choose a random neighbour at rate $1$ and adopt the opinion of that neighbour.
\item
CP: \emph{Contact process}.\\[0.2cm]
This model is defined on $\Omega=\{0,1\}^{\Z^d}$ with rates
$$
c(x,\eta) = \left\{\begin{array}{rl}
\lambda\sum\limits_{y \sim x}\eta(y), &\textrm{if } \eta(x)=0,\\
1, &\textrm{if } \eta(x)=1,
\end{array}\right. \qquad \lambda \geq 0,
$$
which means that infected sites become healthy at rate $1$ and healthy sites become infected at rate $\lambda$ times the number of infected neighbours.
\end{enumerate}

\noindent
It is easy to check that all three models are shift-invariant and attractive. Below we will discuss each model in some detail. We will see that the properties shift-invariant and attractive allow for a number of interesting conclusions concerning their \emph{equilibrium}, as well as their convergence to equilibrium.


\subsection {Partial ordering and coupling}

Given two probability measures $P, P'$ on $\Omega=\{0,1\}^{\mathbb{Z}^d}$ with Borel $\sigma$-algebra $\mathcal{B}$, we say that $P'$ \emph{stochastically dominates} $P$, and write $P \preceq P'$, if 
\[
P(A) \leq P'(A) \text{ for all } A \subseteq \mathcal{B} \text{ non-decreasing,}
\]
where $A$ non-decreasing means 
\[
x \in A \implies A \supseteq \{ y \in \Omega\colon\, x \preceq y \},
\]
and $\preceq$ is a partial ordering on $\Omega$ (in our case the natural partial ordering induced by the ordering of the components at each site). Equivalently, $P'$ stochastically dominates $P$ if 
\[
\int_\Omega f \, \dd P \leq \int_\Omega f \, \dd P' \text{ for all } f\colon\, \Omega \to \R \text{ measurable, bounded and non-decreasing},
\]
where $f$ non-decreasing means 
\[
x \preceq y \implies f(x) \leq f(y).
\]

A \emph{coupling} of $P,P'$ on $\Omega$ is any joint probability measure $\hat{P}$ on $\Omega \times \Omega$ with marginals $P,P'$. Strassen's theorem (see Lindvall \cite{L92}) says that if $P \preceq P'$, then there exists a coupling $\hat{P}$ of $(P, P')$ such that 
\[
\hat{P} \{ (x, x') \in \Omega \times \Omega\colon\, x \preceq x' \} = 1.
\] 
In fact, on a partially ordered space $\Omega$ the following three statements are equivalent:
\begin{enumerate}
\item $P \preceq P'$,
\item $\int_\Omega f \, \dd P \leq \int_\Omega f \, \dd P' \text{ for all } f \text{ measurable, bounded and non-decreasing}$,
\item $\exists\, \hat{P} \colon\, \hat{P} (\hat{X} \preceq \hat{X'}) = 1$,
\end{enumerate}
where $X,X'$ have marginal laws $P,P'$ and $(\hat{X},\hat{X'})$ has joint law $\hat{P}$. 


\subsection{Convergence to equilibrium}

Write $[0]$ and $[1]$ to denote the configurations $\eta\equiv0$ and $\eta\equiv1$, respectively. These are the smallest, respectively, the largest configurations in the natural \emph{partial order} on $\Omega$, and hence
$$
[0] \preceq \eta \preceq [1], \qquad \eta\in\Omega.
$$
Since the dynamics preserves the partial order (see below), we obtain information about what happens when the system starts from any $\eta\in\Omega$ by comparing with what happens when it starts from $[0]$ or $[1]$. 

An IPS can be described by a \emph{semigroup of transition kernels} 
$$
(P_t)_{t\geq0}.
$$
Formally, $P_t$ is an operator acting on $C_b(\Omega)$, the space of bounded continuous functions on $\Omega$, as
$$
(P_tf)(\eta) = \expec_\eta[f(\xi_t)], \qquad \eta \in \Omega, \, f \in C_b(\Omega).
\vspace{0.3cm}
$$
If this definition holds on a \emph{dense} subset of $C_b(\Omega)$, then it \emph{uniquely} determines $P_t$. Note that $P_0$ is the identity and that $P_{t+s} = P_t\circ P_s $ for all $s,t \geq 0$ (where $\circ$ denotes composition). To see the latter, note that 
\[ 
(P_0 f)(\eta) = \mathbb{E}^\eta[f(\xi_0)] = f(\eta),
\]
where $\mathbb{P}^\eta$ is the law of $\xi$ with initial position $\eta$: 
\[
\mathbb{P}^\eta(\xi_\cdot\in C_b(\Omega)\colon\, \xi_0=\eta)=1.
\]
Moreover, for $t \geq 0$ let $\mathcal{F}_t$ be the $\sigma$-algebra generated by $(\xi_s)_{0\leq s\leq t}$. Then
\[
\begin{aligned}
(P_{t+s}f)(\eta) 
&= \mathbb{E}^\eta[f(\xi_{t+s})] = \mathbb{E}^\eta[\mathbb{E}^\eta[f(\xi_{t+s})] \mid \mathcal{F}_t]
= \mathbb{E}^\eta[\mathbb{E}^{\xi_t}[f(\xi_{s})]]\\
&= \mathbb{E}^\eta[(P_t f)(\xi_{s})] = ((P_t\circ P_s) f) (\eta)\,,
\end{aligned}
\]
where the third equality uses the Markov property of $\xi$.

Formally, we can write $P_t = \e^{tL}$ with $L$ the \emph{generator} of the IPS:
$$
(Lf)(\eta) = \sum_{x\in\Z^d} c(x,\eta)[f(\eta^x)-f(\eta)].
$$ 
Alternatively, the semigroup can be viewed as acting on the space of probability measures $\mu$ on $\Omega$ via the \emph{duality relation}
$$
\int_\Omega f\,\d(\mu P_t) = \int_\Omega (P_tf)\,\d\mu, \qquad f \in C_b(\Omega). 
$$

\begin{lemma}
\label{lem:1.1}
Let $(P_t)_{t\geq0}$ denote the semigroup of transition kernels that is associated with $\xi=(\xi_t)_{t \geq 0}$. Write $\delta_\eta P_t$ to denote the law of $\xi_t$ conditional on $\xi_0=\eta$ (which is a probability distribution on $\Omega$). Then
$$
\begin{array}{rl}
t\mapsto \delta_{[0]}P_t & \textrm{is stochastically increasing},\\
t\mapsto \delta_{[1]}P_t & \textrm{is stochastically decreasing}.
\end{array}
$$
\end{lemma}

\begin{proof}
For $t,h\geq0$,
$$
\begin{array}{l}
\delta_{[0]}P_{t+h} = (\delta_{[0]}P_h)P_t \succeq \delta_{[0]}P_t,\\
\delta_{[1]}P_{t+h} = (\delta_{[1]}P_h)P_t \preceq \delta_{[1]}P_t,
\end{array}
$$
where we use that
$$
\delta_{[0]}P_h\succeq \delta_{[0]}, \qquad \delta_{[1]}P_h\preceq\delta_{[1]}, \qquad h \geq 0,
$$ 
and we use Strassen's theorem in combination with the coupling representation that goes with the partial order. (Strassen's theorem says that stochastic ordering is equivalent to the existence of an ordered coupling.)
\end{proof}

\begin{corollary}
\label{cor:1.2}
Both
$$
\begin{aligned}
\underline{\nu}
= \lim_{t\to\infty}\delta_{[0]}P_t 
& \textrm{ = lower stationary law},\\
\overline{\nu} 
= \lim_{t\to\infty}\delta_{[1]}P_t 
& \textrm{ = upper stationary law},
\end{aligned}
$$
exist as probability distributions on $\Omega$ and are equilibria for the dynamics. Any other equilibrium $\pi$ satisfies $\underline{\nu}\preceq\pi\preceq\overline{\nu}$.
\end{corollary}

\begin{proof}
This is immediate from Lemma \ref{lem:1.1} and the sandwich 
$$
\delta_{[0]} P_t\preceq \delta_{\eta}P_t\preceq\delta_{[1]}P_t,
\qquad \eta\in\Omega,\,t \geq 0.
$$
\end{proof}

The class of all equilibria for the dynamics is a convex set in the space of signed bounded measures on $\Omega$. An element of this set is called \emph{extremal} when it is not a proper linear combination of any two distinct elements in the set, i.e., is not of the form 
$$
p\nu_1 + (1-p)\nu_2, \qquad p\in(0,1),\,\nu_1\not=\nu_2.
$$

\begin{lemma}
\label{lem:1.3}
Both $\loweq$ and $\upeq$ are extremal.
\end{lemma}

\begin{proof}
We give the proof for $\upeq$ only. Suppose that 
$$
\upeq=p\nu_1 + (1-p)\nu_2, \qquad \nu_1\neq\nu_2,\,p \in (0,1).
$$
Since $\nu_1$ and $\nu_2$ are equilibria, Corollary \ref{cor:1.2} gives
$$
\int_\Omega f\d\nu_1 \leq \int_\Omega f\d\upeq, 
\qquad \int_\Omega f\d\nu_2 \leq \int_\Omega f\d\upeq,
$$
for any $f$ non-decreasing. Since
$$
\int_\Omega f\d\upeq = p\int_\Omega f\d\nu_1+ (1-p)\int_\Omega f\d\nu_2
$$
and $p\in(0,1)$, both inequalities must be equalities. Integrals of non-decreasing functions determine the measure w.r.t.\ which is being integrated, and so it follows that $\nu_1=\upeq=\nu_2$.

To see why the family $\{\int_\Omega f\,\dd\nu\colon\, f\colon\,\Omega\to\R \text{ non-decreasing}\}$ determines the measure $\nu$, it suffices to consider functions of the form $f(\eta) = 1_{\eta\supseteq \xi}$, $\eta\in\Omega$, with $\xi$ running over $\Omega$. For such functions $\int_\Omega f\,\dd\nu =\nu(A_\xi)$ with $A_\xi = \{\eta\in\Omega\colon\,\eta\supseteq \xi\}$ a non-increasing subset of $\Omega$, and so the claim follows from the monotone class theorem, which says that the sigma-algebra of all events is the same as the sigma-algebra of monotone events.
\end{proof}

\begin{corollary}
\label{cor:1.4}
The following three properties are equivalent (for shift-invariant and attractive spin-flip systems):
\begin{enumerate}
\item[1.] 
$\xi$ is ergodic (i.e., $\delta_{\eta}P_t$ has the same limiting distribution as $t\to\infty$ for all $\eta$).
\item[2.] 
There is a unique stationary distribution.
\item[3.] 
$\loweq=\upeq$.
\end{enumerate}
\end{corollary}

\begin{proof}
The claim is obvious in view of the sandwich of the configurations between $[0]$ and $[1]$.
\end{proof}

\noindent
REMARK: If $\loweq\neq\upeq$, then there is no guarantee that $\lim_{t\to\infty} \mu P_t=\nu$ exists for arbitrary $\mu$. In fact, stronger assumptions than attractiveness are needed to make that happen. We do know that any convergent subsequence has a limit $\nu$ such that $\loweq \preceq \nu \preceq \upeq$.    


\subsection{Properties of the three examples}


\paragraph{Example 1: Stochastic Ising Model.}

For $\beta=0$ , $c(x,\eta)=1$ for all $x$ and $\eta$, in which case the dynamics consists of independent spin-flips, up and down at rate $1$. In that case $\upeq=\loweq=(\tfrac12\delta_{-1}+\tfrac12\delta_{+1})^{\otimes\Z^d}$. For $\beta>0$, the dynamics has a tendency to align the spins. For small $\beta$ this tendency is weak, for large $\beta$ it is strong. It turns out that in $d\geq2$ there is a \emph{critical value} $\beta_d\in(0,\infty)$ such that
$$
\begin{array}{lll}
\beta\leq \beta_d\colon &\loweq=\upeq,\\
\beta>\beta_d\colon     &\loweq\not=\upeq.
\end{array}
$$
The proof uses the so-called Peierls argument. In the first case there is a \emph{unique ergodic equilibrium}, which depends on $\beta$ and is denoted by $\nu_{\beta}$. In the second case there are \emph{two extremal equilibria}, both of which depend on $\beta$ and are denoted by
$$
\begin{array}{llll}
\nu^+_\beta &= \textrm{ plus state} &\textrm{with}  &m^+_\beta = \int_\Omega \eta(0)\nu^+_\beta(\d\eta)>0,\\[0.4cm]
\nu^-_\beta &= \textrm{ minus-state} &\textrm{with} &m^-_\beta = \int_\Omega \eta(0)\nu^-_\beta(\d\eta)<0,
\end{array}
$$
which are called the magnetised states. Note that $\nu^+_\beta$ and $\nu^-_\beta$ are images of each other under the swapping of $+1$'s and $-1$'s and so $m^+_\beta = - m^-_\beta = m_\beta$. 

It can be shown that in $d=2$ all equilibria are a convex combination of $\nu^+_\beta$ and $\nu^-_\beta$, while in $d\geq3$ \emph{other equilibria are possible} as well (e.g.~not shift-invariant) when $\beta$ is large enough. It turns out that $\beta_1=\infty$, i.e., in $d=1$ the SIM is ergodic for all $\beta>0$. It is known that $\beta_2=\tfrac12\log(1+\sqrt{2})$.


\paragraph{Example 2: Voter Model.}

Note that $[0]$ and $[1]$ are both \emph{traps} for the dynamics (if all sites have the same opinion, then no change of opinion occurs), and so
$$
\loweq=\delta_{[0]}, \qquad \upeq=\delta_{[1]}.
$$
It turns out that in $d=1,2$ these are the only extremal equilibria, while in $d\geq 3$ there is a \emph{$1$-parameter family of extremal equilibria}
$$
(\nu_\rho)_{\rho\in[0,1]}
$$
with $\rho$ the density of $1$'s, i.e., $\nu_\rho(\eta(0)=1)=\rho$. This fact is remarkable because the VM has no parameter. For $\rho=0$ and $\rho=1$ these equilibria coincide with $\delta_{[0]}$ and $\delta_{[1]}$, respectively.

\medskip\noindent
REMARK: The dichotomy $d=1,2$ versus $d\geq3$ is directly related to simple random walk being recurrent in $d=1,2$ and transient in $d\geq3$. This property has to do with the fact that the VM is \emph{dual} to a system of coalescing random walks.


\paragraph{Example 3: Contact Process.}

Note that $[0]$ is a trap for the dynamics (if all sites are healthy, then no infection will ever occur), and so
$$
\loweq = \delta_{[0]}.
$$
For small $\lambda$ infection is transmitted slowly, for large $\lambda$ rapidly. It turns out that in $d \geq 1$ there is a \emph{critical value} $\lambda_d\in(0,\infty)$ such that
$$
\begin{array}{lll}
\lambda\leq \lambda_d\colon &\upeq=\delta_{[0]} &\textrm{= extinction, no epidemic},\\
\lambda>\lambda_d\colon &\upeq\not=\delta_{[0]} &\textrm{= survival, epidemic}.
\end{array}
$$

\begin{lemma}
\label{lem:1.5} {\rm (Liggett \cite{L85}, Durrett \cite{D88}.)}\\
(i) $2d\lambda_d\geq 1$.\\
(ii) $d\lambda_d \leq \lambda_1$.\\
(iii) $\lambda_1<\infty$.
\end{lemma}

\noindent
Note that (i--iii) combine to yield that $0<\lambda_d<\infty$ for all $d\geq 1$, so that the phase transition occurs at a \emph{non-trivial} value of the infection rate parameter. Here is a proof of (i) and (ii).

\begin{proof}
(i) Pick $A_0$ finite and consider the CP in dimension $d$ with parameter $\lambda$ starting from the set $A_0$ as the set of infected sites. Let $A = (A_t)_{t \geq 0}$ with $A_t$ the set of infected sites at time $t$. Then
\[
\begin{aligned}
&|A_t| \text{ decreases by } 1 \text{ at rate } |A_t|,\\
&|A_t| \text{ increases by } 1 \text{ at rate } \leq 2d\lambda|A_t|,
\end{aligned}
\]
where the latter holds because each site in $A_t$ has at most $2d$ non-infected neighbours. Now consider the two random processes $X = (X_t)_{t \geq 0}$ with $X_t = |A_t|$ and $Y = (Y_t)_{t \geq 0}$ given by the birth-death process on $\mathbb{N}_0$ that moves at rate $n$ from $n$ to $n-1$ (death) and at rate $(2d\lambda)n$ from $n$ to $n+1$ (birth), both starting from $n_0 = |A_0|$. Then $X$ and $Y$ can be coupled such that
\[
\hat{P}(X_t \leq Y_t \ \forall\, t \geq 0) = 1,
\]
where $\hat{P}$ denotes the coupling measure. Note that $n=0$ is a trap for both $X$ and $Y$. If $2d\lambda < 1$, then this trap is hit with probability $1$ by $Y$, i.e., $\lim_{t \to \infty} Y_t = 0$ a.s., and hence also by $X$, i.e., $\lim_{t \to \infty} X_t = 0$ a.s. Therefore $\nu_\lambda = \delta_0$ when $2d\lambda < 1$. Consequently, $2d\lambda_d \geq 1$.

\medskip\noindent
(ii)
The idea is to couple two CP's that live in dimensions $1$ and $d$. Again, let $A = (A_t)_{t \geq 0}$ with $A_t$ the set of infected sites at time $t$ of $\text{CP}_d(\lambda)$, the CP in dimension $d$ with parameter $\lambda$, this time starting from $A_0 = \{0\}$. Let $B = (B_t)_{t \geq 0}$ be the same as $A$, but for $\text{CP}_1(\lambda d)$, the CP in dimension $1$ with parameter $\lambda d$, starting from $B_0 = \{0\}$.

Define the projection $\pi_d\colon\,\mathbb{Z}^d \to \mathbb{Z}$ as
\[
\pi_d(x_1, \ldots, x_d) = x_1 + \cdots + x_d.
\]
We will construct a coupling $\hat{P}$ of $A$ and $B$ such that
\[
\hat{P}(B_t \subseteq \pi_d(A_t) \ \forall\, t \geq 0) = 1.
\]
From this we get
\[
P(A_t \neq \emptyset \ | \ A_0 = \{0\}) = P\big(\pi_d(A_t) \neq \emptyset \ | \ A_0 = \{0\}) 
\geq P(B_t \neq \emptyset \ | \ B_0 = \{0\}\big),
\]
which implies that if $A$ dies out, then also $B$ dies out. In other words, if $\lambda \leq \lambda_d$, then $\lambda d \leq \lambda_1$, which implies that $d\lambda_d \leq \lambda_1$ as claimed. The construction of the coupling is as follows. Fix $t \geq 0$. Suppose that $B_t \subseteq \pi_d(A_t)$. For each $y \in B_t$ there is at least one $x \in A_t$ with $y = \pi_d(x)$. Pick one such $x$ for every $y$ (e.g. choose the closest up or the closest down). Now couple:
\begin{itemize}
\item If $x$ becomes healthy, then $y$ becomes healthy too.
\item If $x$ infects any of the $d$ sites $x - e_i$ with $i = 1, \ldots, d$, then $y$ infects $y - 1$.
\item If $x$ infects any of the $d$ sites $x + e_i$ with $i = 1, \ldots, d$, then $y$ infects $y + 1$.
\end{itemize}

\medskip\noindent
It is much harder to prove (iii). One way is to compare the CP with directed percolation in two dimensions (= space $\times$ time). See Durrett \cite{D88} for details.  
\end{proof}

\medskip\noindent
REMARK: Sharp estimates are available for $\lambda_1$, but these require heavy machinery. Numerically, $\lambda_1 \approx 1.6494$. A series expansion of $\lambda_d$ in powers of $1/2d$ is known up to several orders, but again the proof is very technical.


\subsection{The Cox-Greven finite systems scheme}

As a prelude to Lectures 2--4, in which we take a closer look at  SIM, VM, CP on \emph{finite random graphs}, we describe what is known about these processes on a large \emph{finite torus} in $\Z^d$,
$$
\Lambda_N = [0,N)^d \cap \Z^d, \qquad N \in \N,
$$
endowed with periodic boundary conditions. The behaviour on $\Lambda_N$ is different from that on $\Z^d$. In particular, there is an $N$-dependent \emph{characteristic time scale} $\alpha_N$ on which the process notices that $\Lambda_N$ differs from $\Z^d$, resulting in \emph{different behaviour for short, moderate and long times}. A systematic study was initiated in Cox, Greven \cite{CG90}, Cox, Greven, Shiga \cite{CGS95,CGS98}.

\medskip\noindent
WARNING: The text in the remainder of this section is technical.


\paragraph{$\bullet$ SIM ON THE TORUS.}

Since $|\Lambda_N|<\infty$, we have
$$
\loweq^N = \upeq^N = \nu^N_\beta \quad \text{ with} \quad \int_\Omega \eta(0)\nu^N_\beta(\dd \eta) = 0 
\qquad \forall\,\beta \in (0,\infty),
$$
i.e., on any finite lattice eventually the average magnetisation vanishes. An interesting question is: How long does it take the SIM to loose its magnetisation and what does it do along the way?

Let
$$
\mathcal{M}^N_t = \frac{1}{|\Lambda_N|} \sum_{x \in \Lambda_N} \xi^N_t(x)
$$
denote the magnetisation at time $t$. Suppose that the law of $\xi^N_0$ is the restriction to $\Lambda_N$ of the equilibrium measure $\nu^-_\beta$ on $\Z^d$, which has magnetisation $m^-_\beta$.

\begin{theorem}
\label{thm:1.6}
{\rm (Cox, Greven \cite{CG90}, Bovier, Eckhoff, Gayrard, Klein \cite{BEGK02}.)}\\
(a) For $\beta<\beta_d$ and any $T_N \to \infty$,
$$
\lim_{N\to\infty} \cL\big[\mathcal{M}^N_{T_N}\big] = \delta_0.
$$
(b) For $\beta>\beta_d$,
$$
\lim_{N\to\infty} \cL\big[\mathcal{M}^N_{s\alpha_N}\big] = m_\beta^{Z_{s}}, \qquad Z_0 = -,
$$
where $(Z_s)_{s \geq 0}$ is the Markov chain on $\{-,+\}$ jumping at rate $1$, and $\alpha_N$ is the average crossover time between the magnetisations associated with $\nu^-_\beta$ and $\nu^+_\beta$ on $\Z^d$ restricted to $\Lambda_N$.
\end{theorem} 

\noindent
For $\beta>\beta_d$ it can further be shown that $(\xi^N_{s}\alpha_N)_{s \geq 0}$ converges in distribution to $\nu_\beta^{Z_{s}}$ as $N\to\infty$ .

The computation of $\alpha_N$ is hard and belongs to the area of \emph{metastability}. It is expected that
$$
\alpha_N = \exp\big[\kappa_d(\beta)N^{d-1}(1 + o(1))\big]
$$
with $\kappa_d(\beta)$ the free energy of the so-called \emph{Wulff droplet} of volume $\tfrac12$ in $\R^d$ representing the barrier between $\nu^-_\beta,\nu^+_\beta$. The proof remains a challenge. See Schonmann, Shlosman \cite{SS98} and Bovier, den Hollander \cite[Chapter 22]{BdH15} for more background.


\paragraph{$\bullet$ VM ON THE TORUS.}

Since $|\Lambda_N|<\infty$, we have
$$
\loweq^N=[0]_N, \qquad \upeq^N = [1]_N,
$$
i.e., on any finite lattice eventually consensus is reached. An interesting question is: How long does it take the VM to reach consensus and what does it do along the way?

Let
$$
\mathcal{O}^N_t = \frac{1}{|\Lambda_N|} \sum_{x \in \Lambda_N} \xi^N_t(x)
$$
denote the fraction of $1$-opinions at time $t$. Suppose that the law of $\xi^N_0$ is the restriction to $\Lambda_N$ of a shift-invariant and ergodic probability measure on $\Z^d$ with mean $\theta \in [0,1]$.

\begin{theorem}
\label{thm:1.7}
{\rm (Cox, Greven \cite{CG90}.)}\\ 
(a) For $d=1,2$ and any $T_N \to \infty$,
$$
\lim_{N\to\infty} \cL\big[\mathcal{O}^N_{T_N}\big] = (1-\theta) \delta_0 + \theta \delta_1.
$$
(b) For $d \geq 3$,
$$
\lim_{N\to\infty} \cL\big[\mathcal{O}^N_{s\alpha_N}\big] = Z_s, \qquad Z_0 = \theta,
$$
where $\alpha_N=|\Lambda_N|$ and $(Z_s)_{s \geq 0}$ is the Fisher-Wright diffusion on $[0,1]$ with diffusion constant $1/G_d$, the inverse of the average number of visits to $0$ of simple random walk on $\Z^d$.
\end{theorem}

A heuristic explanation of Theorem~\ref{thm:1.7} is as follows. The VM is \emph{dual} to a system of coalescing random walks, in the sense that the evolution of the genealogy of the opinions in the fomer is the \emph{time reversal} of the latter. Since simple random walk is recurrent on $\mathbb{Z}^d$ for $d=1,2$ and transient for $d\geq 3$, the dichotomy between (a) and (b) is plausible. For duality and its relation to graphical representations of IPS, see Liggett \cite[Chapter I]{L85}.


\paragraph{$\bullet$ CP ON THE TORUS.}

Since $|\Lambda_N|<\infty$, we have
$$
\loweq^N = \upeq^N = [0]_N \qquad \forall\,\lambda \in (0,\infty),
$$
i.e., on a finite lattice every infection eventually becomes extinct, irrespective of the infection rate. An interesting question is the following: Starting from $[1]_N$, how long does it take the CP to reach $[0]_N$? In particular, we want to know the extinction time
$$
\tau_{[0]_N} = \inf\{t \geq 0\colon\, \xi^N_t = [0]_N\}.
$$
We expect this time to grow slowly with $N$ when $\lambda<\lambda_d$ and rapidly with $N$ when $\lambda>\lambda_d$, where $\lambda_d$ is the critical infection threshold for the infinite lattice $\Z^d$ . 

Let
$$
\mathcal{I}^N_t = \frac{1}{|\Lambda_N|} \sum_{x \in \Lambda_N} \xi^N_t(x)
$$
denote the fraction of infected vertices at time $t$. Suppose that $\xi^N_0 = [1]_N$.

\begin{theorem}
\label{thm:1.8}
{\rm (Cox, Greven \cite{CG90}.)}\\
(a) For $\lambda<\lambda_d$ and any $T_N \to \infty$,
$$
\lim_{N\to\infty} \cL\big[\mathcal{I}^N_{T_N}\big] = \delta_0.
$$
(b) For $\lambda>\lambda_d$,
$$
\lim_{N\to\infty} \cL\big[\mathcal{I}^N_{s\alpha_N}\big] = Z_s, \qquad Z_0=1,
$$
where $\alpha_N = \E_{[1]_N}(\tau_{[0]_N})$ and $(Z_s)_{s \geq 0}$ is the Markov chain on $\{0,1\}$ that jumps from $1$ to $0$ at rate $1$ and is absorbed in $0$.
\end{theorem} 

\begin{theorem}
\label{thm:1.9}
{\rm (Durrett, Liu \cite{DL88}, Durrett, Schonmann \cite{DS88}, Mountford \cite{M93,M99}.)}\\  
There exist $C_-(\lambda),C_+(\lambda) \in (0,\infty)$ such that
$$
\begin{aligned}
&\lambda<\lambda_d\colon \qquad \lim_{N\to\infty} 
\frac{\alpha_N}{\log |\Lambda_N|} = C_-(\lambda),\\[0.5cm]
&\lambda>\lambda_d\colon \qquad \lim_{N\to\infty} 
\frac{\log \alpha_N}{|\Lambda_N|} = C_+(\lambda).
\end{aligned}
$$
\end{theorem}

\noindent
In the subcritical phase the extinction time grows \emph{logarithmically} fast with the volume of $\Lambda_N$, while in the supercritical phase it grows \emph{exponentially} fast. This is a rather dramatic dichotomy. Here is a heuristic explanation.

\begin{itemize}
\item
\emph{Subcritical phase}: When $\lambda<\lambda_c$, the infection cannot sustain itself in the long run. Each infected site has a higher tendency of becoming healthy than of becoming infected. Hence, the number of infected sites decreases over time, and the extinction time scales logarithmically with the system size because the infection dies out only when the last infected site has disappeared.
\item
\emph{Supercritical phase}:
When $\lambda>\lambda_c$, the infection can sustain itself and even grow. Each infected site has a higher tendency of becoming infected than of becoming healthy. Hence, the infection can create a large cluster of infected sites that persists for a long time. The extinction time scales exponentially with the system size because the infection dies out only after all sites become healthy in a short straight run. Many attempts are required to achieve this run, all of which have a small success probability. This implies the memoryless property of the extinction time, from which the exponential scaling follows.  
\end{itemize}

Rough polynomial bounds on $\alpha_N$ are available in $d=1$ at $\lambda=\lambda_1$ (Duminil-Copin, Tassion, Teixeira \cite{DCTT18}).


\section{LECTURE 2: The Stochastic Ising Model (SIM)}
\label{sec:L2} 


\subsection{SIM on graphs}

Let $G=(V,E)$ be a finite connected non-oriented graph. Ising spins are attached to the vertices $V$ and interact with each other along the edges $E$ (see Figure~\ref{fig:graph}).

\begin{figure}[htbp]
\vspace{0.4cm}
\begin{center}
\includegraphics[width=.2\textwidth]{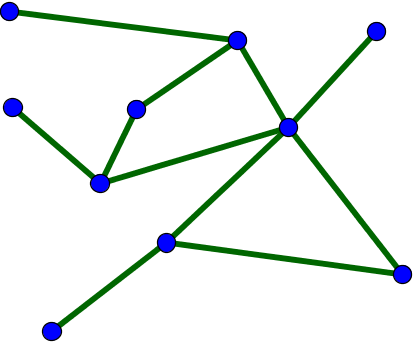}
\end{center}
\vspace{-0.3cm}
\caption{\small A finite connected non-oriented graph.}
\label{fig:graph}
\end{figure}

\noindent
1. The energy associated with the configuration $\sigma = (\sigma_i)_{i \in V} \in \Om = \{-1,+1\}^V$ is given by the Hamiltonian
\[
H(\sigma) = -J \sum_{(i,j) \in E} \sigma_i\sigma_j - h \sum_{i \in V} \sigma_i
\]
where $J>0$ is the \emph{ferromagnetic interaction strength} and $h>0$ is the \emph{external magnetic field}. 

\medskip\noindent
2. Spins flip according to \emph{Glauber dynamics} $(\sigma^G_t)_{t \geq 0}$,
\[
\forall\,\sigma\in\Om \,\,\forall\,j \in V\colon\,\,\,\sigma \to \sigma^j 
\text{ at rate } \ee^{-\beta[H(\sigma^j)-H(\sigma)]_+}
\]
where $\sigma^j$ is the configuration obtained from $\sigma$ by flipping the spin at vertex $j$, and $\beta>0$ is the \emph{inverse temperature}. 

\medskip\noindent
3. The \emph{Gibbs measure} 
\[
\mu(\sigma) = \frac{1}{\Xi}\,\ee^{-\beta H(\sigma)}, \qquad  \sigma \in \Om,
\]
is the reversible equilibrium of this dynamics.

\medskip\noindent
4. Three sets of configurations play a central role (see Figure~\ref{fig:metacar}):
\[
\begin{aligned}
\meta &= \text{metastable state}\\ 
\crit&= \text{crossover state}\\
\stab &= \text{stable state}. 
\end{aligned}
\] 

\begin{figure}[htbp]
\vspace{0.5cm}
\begin{center}
\setlength{\unitlength}{0.5cm}
\begin{picture}(8,6)(.5,2)
\put(0,0){\line(11,0){11}}
\put(0,0){\line(0,9){9}}
\qbezier[30](3.3,3.8)(3.3,1.9)(3.3,0)
\qbezier[50](4.8,5.6)(4.8,3)(4.8,0)
\qbezier[25](6.8,3.0)(6.8,1.5)(6.8,0)
{\thicklines
\qbezier(2,8)(3,2)(4,5)
\qbezier(4,5)(5,7)(6,4)
\qbezier(6,4)(7,2)(8,5)
}
{\thicklines
\qbezier(3.5,5)(4.75,8)(6.5,5)
\qbezier(6.5,5)(6.6,5.1)(6.7,5.2)
\qbezier(6.5,5)(6.4,5.05)(6.3,5.10)
}
\put(3.7,7.5){\tiny{metastable}}
\put(3.7,7){\tiny{crossover}}
\put(2.95,-.9){$\meta$}
\put(4.65,-.9){$\crit$}
\put(6.65,-.9){$\stab$}
\put(11.5,-.15){\tiny{state}}
\put(-1,9.5){\tiny{free energy}}
\put(3.3,4){\circle*{.35}}
\put(4.8,5.75){\circle*{.35}}
\put(6.8,3.20){\circle*{.35}}
\put(3.15,3.85){$\bullet$}
\put(4.65,5.6){$\bullet$}
\put(6.65,3.05){$\bullet$}
\end{picture}
\end{center}
\vspace{1.2cm}
\caption{\small Caricature picture of the free energy landscape [free energy = energy $-$ entropy].}
\label{fig:metacar}
\end{figure}
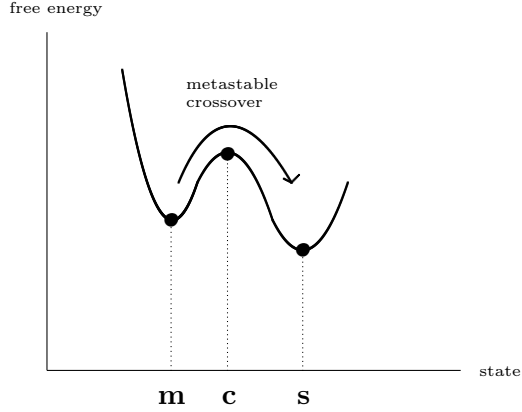

\medskip\noindent
{\bf Definition.}
(a) The stable state is the set of configurations having minimal energy:
\[
\stab = \Big\{\sigma \in \Om\colon\,H(\sigma) = \min_{\zeta\in\Om} H(\zeta)\Big\}.
\]  
(b) The metastable state is the set of configurations not in $\stab$ that lie at the bottom of the next deepest valley:
\[
\meta = \big\{\sigma \in \Om\setminus\stab\colon\,V_\sigma = \max_{\zeta \in \Omega\setminus\stab} V_\zeta\big\}
\]
with $V_\zeta$ the minimal amount a path from $\zeta$ needs to climb in energy in order to reach an energy $<H(\zeta)$.
(c) The crossover state $\crit$ is the set of configurations realising the min-max for paths connecting $\meta$ and $\stab$.
\hfill$\spadesuit$


\subsection{SIM on the complete graph}

Let us see what happens on the complete graph with $N$ vertices. This is a mean-field setting. 

The ferromagnetic interaction strength is chosen to be $J=N^{-1}$. It can be shown that the empirical magnetisation  
\[
m^N_t = \frac{1}{N} \sum_{i \in [N]}  (\sigma^N_t)_i 
\]    
performs a continuous-time random walk on the $2N^{-1}$-grid in $[-1,+1]$, in a potential that is given by the finite-volume free energy per vertex
\[
f^N_{\beta,h}(m) = -\tfrac12 m^2-hm+\beta^{-1}I^N(m)
\]
with an entropy term
\[
I^N(m) = -\frac{1}{N} \log\binom{N}{\frac{1+m}{2}N}.
\]
In the limit $N\to\infty$, the empirical magnetisation performs a Brownian motion on $[-1,+1]$, in a potential that is given by the infinite-volume free energy per vertex
\[
f_{\beta,h}(m) = -\tfrac12m^2 - hm + \beta^{-1}I(m) 
\]
with
\[
I(m) = \tfrac12(1+m)\log(1+m) + \tfrac12(1-m)\log(1-m),
\]
where a redundant shift by $-\log 2$ is dropped. The above formulas describe what is called the \emph{Curie-Weiss model with Glauber dynamics}. 

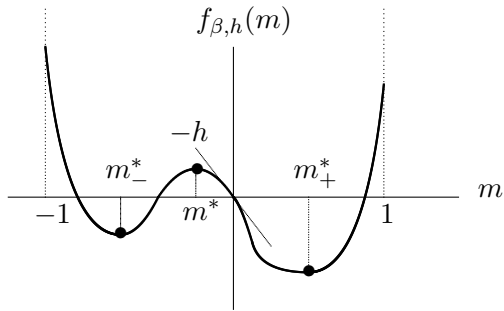
\begin{figure}[htbp]
\begin{center}
\setlength{\unitlength}{0.5cm}
\begin{picture}(10,6)(-4.5,-.2)
\put(-6,0){\line(12,0){12}}
\put(0,-3){\line(0,7){7}}
\qbezier[30](2,0)(2,-1)(2,-2)
\qbezier[30](-3,0)(-3,-.5)(-3,-1)
\qbezier[20](-1,0)(-1,.5)(-1,.75)
\qbezier[50](4,0)(4,2)(4,5)
\qbezier[50](-5,0)(-5,2)(-5,5)
\qbezier[80](1,-1.3)(0,0)(-1,1.3)
{\thicklines
\qbezier(.5,-1.2)(.7,-2)(2,-2)
\qbezier(.5,-1.2)(.2,-.2)(0,0)
\qbezier(2,-2)(3.5,-2)(4,3)
\qbezier(-3,-1)(-2.5,-1)(-2,0)
\qbezier(-3,-1)(-4.5,-1)(-5,4)
\qbezier(0,0)(-1,1.5)(-2,0)
}
\put(6.5,0){$m$}
\put(-1,4.5){$f_{\beta,h}(m)$}
\put(-3.4,.5){$m^*_-$}
\put(1.6,.5){$m^*_+$}
\put(-1.35,-.6){$m^*$}
\put(-1.7,1.5){$-h$}
\put(3.9,-.7){$1$}
\put(-5.3,-.7){$-1$}
\put(-1.15,.56){$\bullet$}
\put(1.82,-2.15){$\bullet$}
\put(-3.19,-1.15){$\bullet$}
\end{picture}
\end{center}
\vspace{0.7cm}
\caption{\small The free energy per vertex $f_{\beta,h}(m)$ at magnetisation $m$ 
(caricature picture with $\meta=m^*_-$, $\crit =  m^*$, $\stab = m^*_+$).} 
\label{fig:doublewell}
\end{figure}

\begin{theorem}
\label{thm:2.1}
{\rm (Bovier, Eckhoff, Gayrard, Klein \cite{BEGK02}.)}\\ 
If $\beta > 1$ and $h \in (0,\chi(\beta))$, then
\[
\E^{\mathrm{CW}}_{\mathbf{m}^-_N}(\tau_{\mathbf{m}^+_N}) 
= K\,\ee^{N\Gamma}[1+o(1)], \qquad N \to \infty,
\]
where $\mathbf{m}^-_N,\mathbf{m}^+_N$ are sets of configurations for which the discrete magnetisations tends to the continuum magnetisations $m^*_-,m^*_+$,
\[
\begin{aligned}
\Gamma &= \beta\,[f_{\beta,h}(m^*)-f_{\beta,h}(m^*_-)]\\[0.3cm]
K &= \pi\beta^{-1} \sqrt{\frac{1+m^*}{1-m^*}\,\frac{1}{1-m^{*2}_-}\,
\frac{1}{[-f_{\beta,h}=''(m^*)]f_{\beta,h}''(m^*_-)}}
\end{aligned}
\]
and
\[
\chi(\beta) = \sqrt{1-\tfrac{1}{\beta}} 
- \tfrac{1}{2\beta} \log\left[\beta\left(1+\sqrt{1-\tfrac{1}{\beta}}\,\right)^{2}\right].
\]
\end{theorem}

The conditions on $\beta,h$ guarantee that $f_{\beta,h}$ has a double-well shape (see Figure~\ref{fig:doublewell}) and represents the parameter regime for which metastable behaviour occurs (see Figure~\ref{fig:metreg}).  

\begin{figure}[htbp]
\vspace{1.2cm}
\begin{center}
\setlength{\unitlength}{.5cm}
\begin{picture}(8,4)(0,0)
\put(0,0){\line(10,0){10}}
\put(0,0){\line(0,5){5}}
{\thicklines
\qbezier(2,0)(4,0)(5,2)
\qbezier(5,2)(6,3.8)(8,3.8)
}
\qbezier[60](0,4)(5,4)(10,4)
\put(10.5,-.4){$\beta$}
\put(-1,5.7){$\chi(\beta)$}
\put(-.25,-1.25){$0$}
\put(-.7,3.7){$1$}
\put(1.8,-1.25){$1$}
\put(1.8,-.25){$\bullet$}
\put(6.3,1.8){\tiny metastable regime}
\end{picture}
\end{center}
\vspace{0.4cm}
\caption{\small Metastable regime for the parameters $\beta,h$.}
\label{fig:metreg}
\end{figure}
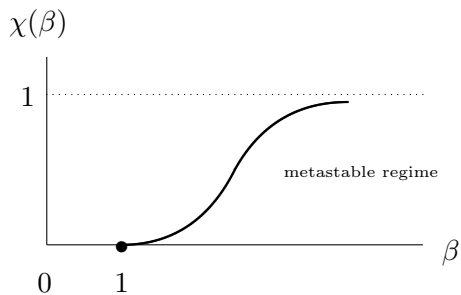

\noindent
The expression for the average crossover time in Theorem \ref{thm:2.1} is called the \emph{Kramers formula}.  


\subsection{SIM on random graphs}

We want to investigate what can be said when the complete graph is replaced by a random graph. Our target will be to derive \emph{Arrhenius laws}, i.e.,
\[
\begin{aligned}
\E_\meta[\tau_\stab] &= K\,\ee^{N\Gamma}[1+o(1)], \qquad N \to \infty,\,\beta \text{ fixed},\\[0.3cm]
\E_\meta[\tau_\stab] &= K\,\ee^{\beta\Gamma}[1+o(1)], \qquad \beta \to \infty, N \text{ fixed}.
\end{aligned}
\]
In general $\Gamma,K$ are \emph{random} and are hard to identify. In fact, in what follows we will mostly have to content ourselves with bounds on these quantities and with convergence in probability under the law of the random graph. In Sections~\ref{ss.HER}--\ref{ss.IER} we focus on \emph{dense} homogeneous and inhomogeneous Erd\H{o}s-R\'enyi random graphs, and derive an Arrhenius law of the first type. In Sections~\ref{ss.sparse}--\ref{ss.CM} we focus on \emph{sparse} graphs, both deterministic and random, and on configuration models, and derive an Arrhenius law of the second type. 


\subsection{SIM on the Erd\H{o}s-R\'enyi random graph}
\label{ss.HER}

\begin{theorem}
\label{thm:2.2}
{\rm (den Hollander, Jovanovski \cite{dHJ21}).}\\
On the Erd\H{o}s-R\'enyi random graph with $N$ vertices, for $J=1/pN$, $\beta > 1$ and $h \in (0,\chi(\beta))$, 
\[
\E^{\mathrm{ER}}_{\mathbf{m}^-_N}(\tau_{\mathbf{m}^+_N}) 
= N^{E_N}\,\E^{\mathrm{CW}}_{\mathbf{m}^-_N}(\tau_{\mathbf{m}^+_N}), 
\qquad N \to \infty,
\]
where $E_N$ is a random exponent that satisfies
\[
\lim_{N\to\infty} \mathbb{P}_{\mathrm{ER_N(p)}}\big(|E_N| \leq \tfrac{11}{6}\tfrac{\beta}{p}(m^*-m_-)\big) = 1,
\]
with $\mathbb{P}_{\mathrm{ER_N(p)}}$ the law of the random graph.
\end{theorem} 

\begin{figure}[htbp]
\begin{center}
\includegraphics[width=.2\textwidth]{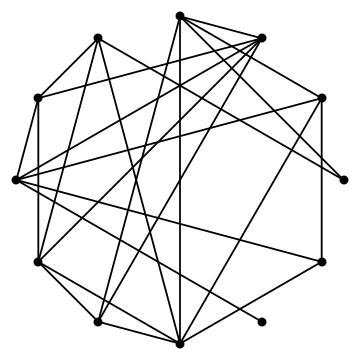}
\end{center}
\vspace{-0.3cm}
\caption{\small Erd\H{o}s-R\'enyi random graph (ERRG): take the complete graph with $N$ vertices and retain edges with probability $p \in (0,1)$.}
\label{fig:ER}
\end{figure}

Apart from a polynomial error term, the crossover time is the same on the Erd\H{o}s-R\'enyi random graph (see Figure~\ref{fig:ER}) as on the complete graph, after the \emph{change of interaction} from $J=1/N$ to $J=1/pN$. The asymptotic estimate of the crossover time is uniform in the starting configuration drawn from the set $\mathbf{m}^-_N$. Note that $J$ needs to be \emph{scaled} up by a factor $1/p$ in order to allow for a comparison with the Curie-Weiss model: in the Erd\H{o}s-R\'enyi model every spin interacts with $\sim pN$ spins rather than $N$ spins. The critical value in equilibrium changes from $1$ to $1/p$ (Bovier, Gayrard \cite{BG93}). On the complete graph the prefactor is constant and computable. On the Erd\H{o}s-R\'enyi random graph it is random and more involved. 

The proof of Theorem \ref{thm:2.2} follows the pathwise approach to metastability (see Bovier, den Hollander \cite{BdH15}). In particular, the empirical magnetisation $(m^N_t)_{t \geq 0}$ is monitored on a mesoscopic space-time scale. The difficulty is that the lumping technique typical for mean-field settings is no longer available: after projection the Markov property is lost. The way around this problem is via coupling: sandwich $(m^N_t)_{t \geq 0}$ between two Curie-Weiss models with a perturbed magnetic field $h^N$, tending to $h$ as $N\to\infty$. The computations are rather elaborate and are beyond the scope of the present overview.

The following theorem provides a \emph{refinement of the prefactor}.

\begin{theorem}
\label{thm:2.3}
{\rm (Bovier, Marello, Pulvirenti \cite{BMP21}).}\\
For $\beta>1$, $h > 0$ small enough and $s>0$,\[
\lim_{N\to\infty} \mathbb{P}_{\mathrm{ER}_N(p)}\left(C_1\ee^{-s} \leq
\frac{\E^{\mathrm{ER}}_{\mathbf{m}^-_N}(\tau_{\mathbf{m}^+_N})}
{\E^{\mathrm{CW}}_{\mathbf{m}^-_N}(\tau_{\mathbf{m}^+_N})} 
\leq C_2\ee^s\right) \geq 1-k_1\ee^{-k_2s^2},
\]
where $k_1,k_2>0$ are absolute constants, and $C_1=C_1(p,\beta)$ and $C_2=C_2(p,\beta,h)$. 
\end{theorem}

\noindent
This theorem shows that the prefactor is a tight random variable, and hence constitutes a considerable sharpening of Theorem \ref{thm:2.2}. The proof of Theorem \ref{thm:2.3} uses the potential-theoretic approach to metastability.

The local homogeneity of the Erd\H{o}s-R\'enyi random graph again plays a crucial role: it turns out that the exact same test functions and test flows that are employed in relevant variational estimates work for the Curie-Weiss model and can be used to give sharp upper and lower bounds on the average crossover time. The better control on the prefactor comes at a price: the magnetic field has to be taken small enough; the dynamics starts from the last-exit biased distribution on $\mathbf{m}^-_N$ for the transition from $\mathbf{m}^-_N$ to $\mathbf{m}^+_N$, rather than from an arbitrary configuration in $\mathbf{m}^-_N$.

Proofs rely on elaborate techniques: isoperimetric inequalities, concentration estimates, capacity estimates, coupling techniques, coarse-graining techniques. These techniques exploit the fact that in the dense regime the Erd\H{o}s-R\'enyi random graph is \emph{locally homogeneous}.


\subsection{SIM on the inhomogeneous Erd\H{o}s-R\'enyi random graph}
\label{ss.IER}

Theorem \ref{thm:2.3} can be extended to the inhomogeneous ERRG. The Hamiltonian becomes
\[
H(\sigma) = - \sum_{(i,j) \in E} J_{ij} \sigma_i\sigma_j - h \sum_{i \in V} \sigma_i
\]
with $J_{ij}>0$ independent random variables. An example is Bernoulli with probability $r(\frac{i}{N},\frac{j}{N})$, where 
$$
r(x,y), \qquad  x,y \in [0,1],
$$ 
is a continuous reference graphon. A special case is the rank-$1$ choice $r(x,y)=v(x)v(y)$ for some weight function $v(x)$, $x \in [0,1]$, which corresponds to the Chung-Lu Random Graph. See Bovier, den Hollander, Marello, Pulvirenti, Slowik \cite{BdHMPS24} for further details.


\subsection{SIM on sparse graphs}
\label{ss.sparse}

The ERRG is a dense graph. We next consider sparse graphs. Given a finite connected non-oriented multigraph
\[
G=(V,E),
\]
the Hamiltonian is
\[
H(\sigma) = -\frac{J}{2} \sum_{(i,j) \in E} \sigma_i\sigma_j
-\frac{h}{2} \sum_{i \in V} \sigma_i, \qquad \sigma\in\Omega,
\]
where $J>0$ is the ferromagnetic pair potential and $h>0$ is the magnetic field. We write $\PP^{G,\beta}_\sigma$ to denote the law of $(\sigma^G_t)_{t \geq 0}$ given $\sigma^G_0=\sigma$. The upper indices $G,\beta$ exhibit the dependence on the underlying graph $G$ and the interaction strength $\beta$ between neighbouring spins.

It is easy to check that $\stab = \{\boxplus\}$ for all $G$ because $J,h>0$. For general $G$, however, $\meta$ is not a singleton, but we will be interested in those $G$ for which the following hypothesis is satisfied (see Figure~\ref{fig:glauber}): 
\[
\text{(H)} \qquad \meta = \{\boxminus\}.
\]
The energy barrier between $\boxminus$ and $\boxplus$ is
\[
\Gamma^\star = H(\cC^\star) - H(\boxminus),
\]
where $\cC^\star=\crit$ is the set of \emph{critical configurations realising the min-max} for the crossover from $\boxminus$ to $\boxplus$, all of which have the same energy.

\begin{figure}[htbp]
\begin{center}
\vspace{2.2cm}
\setlength{\unitlength}{0.45cm}
\begin{picture}(8,6)(0,0)
\put(0,0){\line(11,0){11}}
\put(0,0){\line(0,9){9}}
\qbezier[30](3.3,3.8)(3.3,1.9)(3.3,0)
\qbezier[50](4.8,5.6)(4.8,3)(4.8,0)
\qbezier[25](6.8,3.0)(6.8,1.5)(6.8,0)
\qbezier[40](0,5.8)(2.5,5.8)(4.8,5.8)
{\thicklines
\qbezier(2,8)(3,2)(4,5)
\qbezier(4,5)(5,7)(6,4)
\qbezier(6,4)(7,2)(8,5)
}
\put(3,-1.3){$\boxminus$}
\put(4.55,-1.3){$\cC^\star$}
\put(6.5,-1.3){$\boxplus$}
\put(11.5,-.3){$\xi$}
\put(-1.3,9.6){$H(\xi)$}
\put(-1.8,5.5){$\Gamma^\star$}
\put(3.3,4){\circle*{.35}}
\put(4.8,5.75){\circle*{.35}}
\put(6.8,3.20){\circle*{.35}}
\end{picture}
\end{center}
\vspace{0.2cm}
\caption{\small Schematic picture of $H$ and $\boxminus,\boxplus$ and $\Gamma^\star,\cC^\star$.}
\label{fig:glauber}
\end{figure}
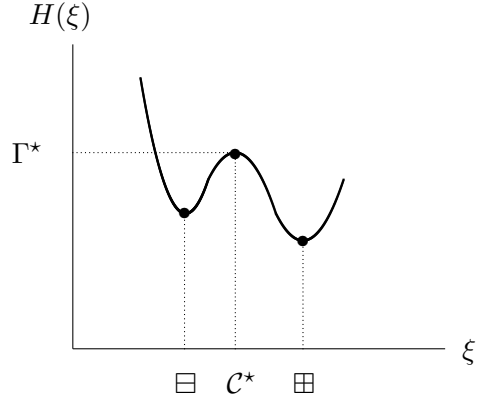

\begin{theorem}
\label{thm:2.4}
{\rm (Bovier, den Hollander \cite{BdH15}).}\\ 
There exists a $K^\star \in (0,\infty)$, called prefactor, such that
\[
\lim_{\beta\to\infty}
\ee^{-\beta\Gamma^\star}\,\E^{G,\beta}_\boxminus(\tau_\boxplus) = K^\star.
\]
\end{theorem}

The validity of Theorem \ref{thm:2.4} does not rely on the details of the graph $G$, provided it is finite, connected and non-oriented. For concrete choices of $G$, the task is to identify the \emph{critical triplet} (see Figure~\ref{fig:glauber})
\[
(\cC^\star,\Gamma^\star,K^\star).
\]
For deterministic graphs this task has been successfully carried out for a large number of examples. However, for random graphs the triple is random, and identification represents a very serious challenge. In what follows we focus on the CM. 


\subsection{SIM on the configuration model}
\label{ss.CM}

The CM is a sparse graph that can be generated via a simple pairing algorithm (see Figure~\ref{fig:pairing}). 

\vspace{0.5cm}
\begin{figure}[htbp]
\begin{center}
\includegraphics[width=.25\textwidth]{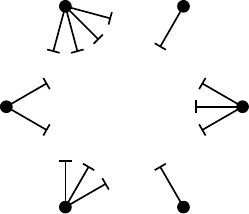}
\hspace{3cm}
\includegraphics[width=.25\textwidth]{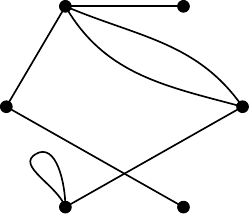}
\end{center}
\caption{\small Configuration model with $6$ vertices and degrees $(1,3,1,3,2,4)$ after randomly pair half-edges.}
\label{fig:pairing}
\end{figure}

\noindent
WARNING: The text in the remainder of this section is technical. We go over it in leaps to sketch the main picture. 

\medskip
In order to state our main theorems, we need some notations and definitions.

\medskip\noindent
1. Fix $N\in\N$. With each vertex $i \in [N]$ we associate a random degree $d_i$, in such a way that 
\[
(d_i)_{i \in [N]}
\] 
are i.i.d.\ with probability distribution $f$ conditional on the event $\{\sum_{i \in [N]} d_i = \mbox{even}\}$. Consider a uniform matching of the half-edges, leading a multi-graph $\CM_N$ satisfying the requirement that the degree of vertex $i$ is $d_i$ for $i \in [N]$. The total number of edges is $\tfrac12\sum_{i \in [N]} d_i$.

\medskip\noindent
2. Throughout the sequel we write $\PP_N$ to denote the law of the random multi-graph $\CM_N$ generated by the Configuration Model.

\medskip\noindent
3. To avoid degeneracies we assume that
\[
\begin{aligned}
\dmin &= \min\{k\in\N\colon\,f(k)>0\} \geq 3,\\
\dave &= \sum_{k\in\N} kf(k) < \infty,
\end{aligned}
\]
i.e., all degrees are at least three and the average degree is finite. In this case $\CM_N$ is connected with high probability ($\whp$), i.e., with probability tending to $1$ as= $N\to\infty$.

\medskip\noindent
4. Along the way we need a technical function that allows us to quantify certain properties of the energy landscape, which we introduce next. Later we provide the underlying heuristics. For $x \in (0,\tfrac12]$ and $\delta \in (1,\infty)$, define (see Figure~\ref{fig:plot})
\[
\begin{aligned}
&I_{\delta}\left(x\right) = \inf\Big\{y \in (0,x]\colon\,1<x{}^{x\left(1-1/\delta\right)}
\left(1-x\right)^{\left(1-x\right)\left(1-1/\delta\right)}\\
&\qquad \qquad \qquad \times \left(1-x-y\right)^{-\left(1-x-y\right)/2}\left(x-y\right)^{-\left(x-y\right)/2}y^{-y}\Big\}.
\end{aligned}
\]

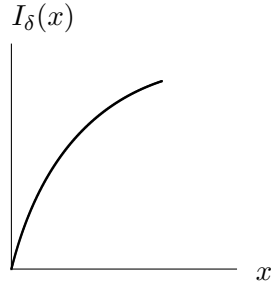
\begin{figure}[htbp]
\vspace{-0.5cm}
\begin{center}
\setlength{\unitlength}{0.5cm}
\begin{picture}(6,6)(0,2)
\put(0,0){\line(6,0){6}}
\put(0,0){\line(0,6){6}}
{\thicklines
\qbezier(0,0)(1,4)(4,5)
}
\put(6.5,-.3){$x$}
\put(0,6.5){$I_\delta(x)$}
\end{picture}
\end{center}
\vspace{0.7cm}
\caption{\small Plot of the function $x \mapsto I_\delta(x)$ for $\delta=6$.}
\label{fig:plot}
\end{figure}


\paragraph{Main theorems.}

The following results are taken from Dommers, den Hollander, Jovanovski, Nardi \cite{DdHJN17}. We want to prove Hypothesis (H) and also to identify the critical triplet for $\CM_N$, which we henceforth denote by $(\cC^\star_N,\Gamma^\star_N,K^\star_N)$, in the limit as $N\to\infty$. 

Our first theorem settles Hypothesis (H) for small $h/J$. Suppose that
\[
\frac{h}{J} < 
\frac{
2I_{d_{\mathrm{ave}}}\left(\tfrac{1}{2}\right)
-\tfrac{1}{2}\left(1-4I_{d_{\mathrm{min}}}\left(\tfrac{1}{2}\right)\right)^{2}
\left(1-2I_{d_{\mathrm{min}}}\left(\tfrac{1}{2}\right)\right)^{-1}
}{
\left(\tfrac{1}{d_{\mathrm{ave}}} + \tfrac{1}{2} \right)
}.
\]

\begin{theorem}
\label{thm:2.5}
If the above inequality is satisfied, then
\[
\lim_{N\to\infty} \PP_N\big(\CM_N \text{ satisfies {\rm (H)}}\big) = 1.
\]
\end{theorem}

Our second and third theorem provide upper and lower bounds on $\Gamma^\star_N$. Label the vertices of the graph in order of increasing degree. Let $\gamma\colon\,\boxminus\to\boxplus$ be the path that successively flips the vertices $1,\ldots,N$ (in that order), and for $M \in [N]$ let $\ell_M = \sum_{i \in [M]} d_ i$.

\begin{theorem}
\label{thm:2.6}
Define
\[
\bar{M} = \bar{M}\left(\frac{h}{J}\right)
= \min\left\{M \in [N]\colon\,\frac{h}{J} \geq
\ell_{M+1}\left(1-\frac{\ell_{M+1}}{\ell_N}\right)-\ell_M\left(1-\frac{\ell_M}{\ell_N}\right)\right\},
\]
and note that $\bar{M} < N/2$. Then with high probability
\[
\Gamma_N^{\star} \leq \Gamma_N^+,
\qquad \Gamma_N^+ = J\ell_{\bar{M}}\Big(1-\frac{\ell_{\bar{M}}}{\ell_N}\Big)
-h\bar{M} \pm O\big(\ell_N^{3/4}\big).
\]
\end{theorem}

\begin{theorem}
\label{thm:2.7}
Define
\[
\tilde{M} = \min\left\lbrace M \in [N]\colon\, \ell_M \geq \tfrac12\ell_N\right\rbrace.
\]
Then with high probability
\[
\Gamma_N^\star \geq \Gamma_N^-, \qquad
\Gamma_N^- = J\,d_\mathrm{ave}\,I_{d_\mathrm{ave}}
\left(\tfrac12\right) N - h\tilde{M} - o(N).
\]
\end{theorem}

\begin{corollary}
\label{cor:2.8}
Under Hypothesis {\rm (H)}, Theorems \ref{thm:2.6}--\ref{thm:2.7} yield 
\[
\lim_{\beta\to\infty} \PP_{\boxminus}^{G,\beta}\left(\ee^{\beta(\Gamma_N^- -\varepsilon)}
\leq \tau_\boxplus \leq \ee^{\beta(\Gamma_N^+ +\varepsilon)}\right) =1 \quad \forall\,\varepsilon>0.
\]
\end{corollary}

\medskip\noindent
REMARK: For simple degree distributions, like Dirac or power law, the quantities $\bar{M}$ , $\ell_{\bar{M}}$ , $\tilde{M}$ can be computed explicitly.

\medskip
The bounds in Theorems \ref{thm:2.6}--\ref{thm:2.7} are tight in the limit of large degrees. Indeed, by the law of large numbers we have that
\[
\ell_{N} \frac{\ell_{\bar{M}}}{\ell_N}
\left(1-\frac{\ell_{\bar{M}}}{\ell_N}\right)
\leq \tfrac{1}{4} \ell_{N} = \tfrac{1}{4} d_{\mathrm{ave}}\,N\left[1+o(1)\right].
\]
Hence
\[
\frac{\Gamma_N^+}{\Gamma_N^-} = \frac{\frac{1}{4} d_{\mathrm{ave}} \left[1+o\left(1\right)\right]
- \frac{h}{J} \frac{\bar{M}}{N}+o(1)}{d_{\mathrm{ave}}
I_{d_{\mathrm{ave}}}\left(\frac{1}{2}\right)-\frac{h}{J}\frac{\tilde{M}}{N}-o(1)}.
\]
In the limit as $d_{\mathrm{ave}} \to \infty$ we have $I_{d_{\mathrm{ave}}}(\frac{1}{2})\to \frac{1}{4}$, in which case the above ratio tends to $1$.


\paragraph{Discussion.}

$\mbox{}$

\medskip\noindent
1. The integer $\bar{M}$ has the following interpretation. The path $\gamma\colon\boxminus\to\boxplus$ is obtained by flipping $(-1)$-valued vertices to $(+1)$-valued vertices in order of increasing degree. Up to fluctuations of size $o(N)$, the energy along $\gamma$ increases for the first $\bar{M}$ steps and decreases for the remaining $N-\bar{M}$ steps.

\medskip\noindent
2. The integer $\tilde{M}$ has the following interpretation. To obtain our lower bound on $\Gamma^{\star}_{N}$ we consider configurations whose $(+1)$-valued vertices have total degree at most $\tfrac12 \ell_{N}$. The total number of $(+1)$-valued vertices in such type of configurations is at most $\tilde{M}$ .

\medskip\noindent
3. If we consider all sets on $\CM_N$ that are of total degree $x\ell_{N}$ and share $y\ell_{N}$ edges with their complement, then $I_{\delta}(x)$ represents (a lower bound on) the least value for $y$ such that the average number of such sets is at least $1$. In particular, for smaller values of $y$ this average number is exponentially small.

\medskip\noindent
4. We believe that Hypothesis (H) holds as soon as
\[
0 < h < (d_\mathrm{min}-1)J,
\]
i.e., we believe that in the limit as $\beta\to\infty$ followed by $N\to\infty$ this choice of parameters corresponds to the metastable regime of our dynamics, i.e., the regime where $(\boxminus,\boxplus)$ is a metastable pair.

\medskip\noindent
5. The scaling behaviour of $\Gamma_N^\star,K_N^\star$ as $N\to\infty$, as well as the geometry of $\cC_N^\star$, are hard to capture. Here are some conjectures put forward in Dommers, den Hollander, Jovanovski, Nardi \cite{DdHJN17}. 

\begin{conjecture}
\label{conj:2.9}
There exists a $\gamma^\star \in (0,\infty)$ such that
\[
\lim_{N\to\infty} \PP_N\Big( \big| N^{-1} \Gamma^\star_N - \gamma^\star\big| > \delta\Big) = 0 \qquad \forall\,\delta>0.
\]
\end{conjecture}

\begin{conjecture}
\label{conj:2.10}
There exists a $c^\star \in (0,1)$ such that
\[
\lim_{N\to\infty} \PP_N\Big( \big| N^{-1} \log|\cC^\star_N| - c^\star\big| > \delta\Big) = 0 \qquad \forall\,\delta>0.
\]
\end{conjecture}

\begin{conjecture}
\label{conj:2.11}
There exists a $\kappa^\star \in (1,\infty)$ such that
\[
\lim_{N\to\infty} \PP_N\Big( \big| |\cC^\star_N|\,K^\star_N - \kappa^\star\big| > \delta\Big) = 0 \qquad \forall\,\delta>0.
\]
\end{conjecture}

\medskip\noindent
6. It is shown in Dommers \cite{D17} that for a random regular graph with degree $r \geq 3$, there exist constants $0<\gamma_-^\star(r)<\gamma_+^\star(r)<\infty$ such that
\[
\lim_{N\to\infty} \lim_{\beta\to\infty}
\E_N\left( \PP^{\CM_N}_\boxminus\left(\ee^{\beta N\gamma_-^\star(r)}
\leq \tau_\boxplus \leq \ee^{\beta N\gamma_+^\star(r)}\right) \right) = 1
\]
when $\frac{h}{J} \in (0,C_0\sqrt{r})$ for some constant $C_0 \in (0,\infty)$ that is small enough. Moreover, there exist constants $C_1 \in (0,\tfrac14\sqrt{3})$ and $C_2 \in (0,\infty)$ (depending on $C_0$) such that
\[
\gamma_-^\star(r) \geq \tfrac14 Jr - C_1J\sqrt{r}, \qquad
\gamma_+^\star(r) \leq \tfrac14 Jr + C_2J\sqrt{r}.
\]
These results are derived without Hypothesis (H), but it is shown that Hypothesis (H) holds as soon as $r \geq 6$ .


\section{LECTURE 3: The Voter Model (VM)}
\label{sec:L3} 

In this lecture we focus on the VM on the \emph{regular random graph}. We analyse how the fraction of discordant edges evolves over time, in the limit as the size of the graph tends to infinity, on \emph{three time scales}: short, moderate, and long. We also analyse what happens when the edges of the random regular graph are \emph{randomly rewired} while the VM is running. It will turn out that the graph dynamics has several interesting consequences. Most of what is written below is taken from Avena, Baldasso, Hazra, den Hollander, Quattropani \cite{ABHdHQ22,ABHdHQ24}.

Given a connected graph $G=(V,E)$, the voter model is the Markov process $(\xi_t)_{t \geq 0}$ on state space $\{0,1\}^V$ where each vertex carries opinion $0$ or $1$, at rate $1$ selects one of the neighbouring vertices uniformly at random, and adopts its opinion. Write $\xi_t = \{\xi_t(i)\colon\,i\in V\}$ with $\xi_t(i)$ the opinion at time $t$ of vertex $i$. We analyse the evolution of the \emph{fraction of discordant edges} 
\[
\cD^N_t = \frac{|D^N_t|}{M}, \qquad D^N_t = \big\{(i,j) \in E\colon\, \xi_t(i) \neq \xi_t(j)\big\},
\]
where $N= |V|$ and $M=|E|$. This is an interesting quantity because it monitors the \emph{size of the boundary between the two opinions}.

The \emph{consensus time} is defined as
\[
\tau_{\rm cons} = \inf\{t \geq 0\colon\, \xi_t(i)=\xi_t(j)\, \forall\, i,j\in V\}.
\]
For finite graphs we know that $\tau_{\rm cons}<\infty$ with probability $1$, either at $[0]_N$ or at $[1]_N$. The interest lies in determining the \emph{relevant time scale} on which consensus is reached, and \emph{how} it is reached. Via time reversal, the voter model is \emph{dual} to a system of coalescing random walks, describing the genealogy of the opinions, as shown in the next section. This section is an \emph{intermezzo}, after which we will return to the random regular graph.


\subsection{VM duality and graphical representation}
\label{ss.dual}

In this section we give the \emph{graphical representation} of the VM and indicate how \emph{duality} is obtained via \emph{time reversal}. One of the reasons why voter models have been studied intensively is the fact that they represent a class of interacting particle systems for which the dual is simple, namely, a \emph{system of coalescing random walks}. This allows for a rephrasing of problems regarding voter models as problems regarding systems of coalescing random walks, which are often easier to deal with.

There are two ways to define a duality relation between two Markov processes $X = (X_t)_{t \geq 0}$ and $Y = (Y_t)_{t \geq 0}$ with state space $\mathcal{X}$ and $\mathcal{Y}$, respectively: \emph{analytically} or \emph{graphically}. The latter considers the time reversal of the graphical representation of the process and deduces information of the original process by using the evolution of the reversed process (which is often easier to study). The former needs the identification of a bounded measurable function $H\colon\,\mathcal{X} \times \mathcal{Y} \to \R$ such that 
\[ 
\E^x[H(X_t,y)] = \E^y[H(x,Y_t)], \qquad x \in \mathcal{X}, \,y \in \mathcal{Y},\,t \geq 0.
\]
If this is the case, then we say that $X$ and $Y$ are dual to one another with respect to the duality function $H$. Below we give a summary of both forms of duality. An important reference for interacting particle systems on graphs, particularly for the duality between the voter model and coalescing random walks on graphs, is \cite[Chapter 14]{AF02}.


\paragraph{Definition of the general voter model.}

Let us first recall the definition of the voter model in terms of its transition rates. We consider the continuous-time voter model $(\eta_t)_{t \geq 0}$ with state space $\XX = W^V$, where $V = [n] = \{1,\dots,n\}$, $n\in\N$, is the vertex set of a finite graph $G$, while $W$ is a finite alphabet of admissible opinions with $|W| = o \in \{2,\dots,n\}$. We will be mainly interested in the case $o=2$, with a given initial configuration of opinions $\eta_0 \in \XX$, and in the case $o=n$, which corresponds to the setting in which every site has a different opinion. For each current state $\eta \in \XX$, the only allowed transitions are the ones to the states $\eta^{x \leftarrow y} \in \XX$ defined as
\[
\eta^{x\leftarrow y} (z) =
\begin{cases}
\eta(y) & \text{if } z=x, \\
\eta(z) & \text{if } z\neq x,
\end{cases}
\]
for some $x,y \in [n]$ such that $\eta(y) \neq \eta(x)$. These transitions describe the events in which, starting from a configuration $\eta$, the voter at site $x$ adopts the opinion of the voter at site $y$, while all other sites retain their opinions. We can write the generator $L$ of the process as follows: for any $f \in D(\XX)$ (= the set of Lipschitz functions on $\XX$) and $\eta\in\XX$,
\[
(L f)(\eta) 
= \sum_{\eta'\in\XX} \omega(\eta,\eta') \left[f(\eta')- f(\eta)\right]
= \sum_{x\in[n]}\sum_{y\in [n]} c_y(x,\eta) \left[f(\eta^{x\leftarrow y})- f(\eta)\right],
\]
where the values $\omega(\eta,\eta') \geq 0$ represent the rates at which transitions $\eta \to \eta'$ occur for $\eta,\eta' \in \XX$. The only possible transitions are those of the type $\eta\to\eta^{x \leftarrow y}$ for some sites $x,y$ with different opinions.

Thus, we renamed the sum over the configurations $\eta'$ of $\omega(\eta,\eta^{x\leftarrow y})$ as the sum over the sites and opinions of the rates $c_y(x,\eta)$, emphasising the dependence on the parameters $x,y,\eta$. Since we want that, for every pair of sites $x,y$, transitions $\eta\to\eta^{x\leftarrow y}$ occur with a rate proportional to the number of neighbours $y$ of $x$ having opinion $\eta(y)$, we set 
\[
c_y(x,\eta) = \sum_{z \in [n]} p(x,z)\,1_{\{\eta(z) = \eta(y),\eta(x) \neq \eta(y)\}},
\]
where $p(\cdot,\cdot)$ are the jump probabilities of an irreducible continuous-time random walk on $[n]$, in particular, $p(x,y) \geq 0$ for all $x,y \in [n]$ and $\sum_{y \in [n]} p(x,y)=1$ for all $x \in [n]$. In other words, a site $x$ waits an exponential time of parameter $1$, after which it adopts the value of the voter seen at a site $z$ that is chosen with probability $p(x,z)$. It follows immediately that the underlying random walk structure given by $p(\cdot,\cdot)$ uniquely determines the process. We mention that in the case where $W=\{0,1\}$, due to the invariance under relabelling of the two opinions, the rates can be written as $c(x,\eta) = \sum_{y\colon\, \eta(x) \neq \eta(y)} p(x,y)$, and represent the rate of the transition $\eta\to\eta_x$, where $\eta_x(z) = 1-\eta(x)$ if $z = x$ and $\eta_x(z) = \eta(z)$ otherwise. 


\paragraph{Analitic duality.}

In order to state the duality relation analytically, let us for the moment consider the case in which $W = \{0,1\}$, i.e., the voter model over $X = W^{[n]}$ is a spin-flip system. Once the duality relation is stated in this setting, the natural generalisation to a finite number of opinions $o$ will be straightforward. For a detailed discussion of duality for spin-flip systems, we refer the reader to Liggett \cite{L85, L99}. Let us denote by $(A_t)_{t \geq 0}$ the dual Markov process with respect to the voter model $(\eta_t)_{t \geq 0}$ with state space $X$. Since we are considering a spin-flip system, the state space of the dual $(A_t)_{t \geq 0}$ is taken to be
\[
Y=\{A\colon\, A \text{ finite subset of } [n]\},
\]
which is finite for every $n\in\N$. Therefore the dual process $(A_t)_{t \geq 0}$ is actually a Markov chain on $Y$, and can be interpreted as the locations in time of a collection of independent continuous-time random walks on $[n]$ that coalesce every time two of them meet at the same site. It follows that $|A_t|$ can only decrease as $t$ increases. The duality function $H$ that will take this into account is 
\[
H(\eta,A) = 1_{\{\eta(x)=1\,\,\forall x\in A\}}, \qquad \eta\in X, \,A\in Y.
\]
The chief reason for this choice is related to the fact that we want to look at consensus states, i.e., $\eta$'s such that $\eta(x) = \eta(y)$ for all $x,y \in [n]$. It is not difficult to see that, by using the latter duality function in combination with the generator and the coefficients for $k=2$, we get that the rates of $(A_t)_{t \geq 0}$ for the transitions $A \to B$, $A,B\in Y$, are
\[
q(A,B) = \sum_{x\in A} \sum_{\substack{y\in [n]:\\ (A\setminus \{x\})\cup \{y\} =B}} p(x,y).
\]
The interpretation is the following. Each $x\in A$ is removed from $A$ at rate $1$ and is replaced by $y$ with probability $p(x,y)$. Moreover, when an attempt is made to place a point $y$ at a site that is already occupied, then the two points coalesce into one. In our system of coalescing random walks, this means that each random walk independently has an exponential waiting time of parameter $1$, moves according to $p(x,y)$, and when moving to an occupied site coalesces with the occupier. Thus, the semigroup of each random walk is given by
\[
P^t(x,y) = \ee^{-t} \sum_{n \in \N_0} \frac{t^n}{n!}\, p^n(x,y), \qquad x,y \in [n], t \geq 0.
\]

All the previous statements can be generalised to a setting with $k$ opinions with $k \in \{2,\dots,n\}$ in the following way. The dual state space is taken to be $Y^{(k)} = Y^{k-1}$, the product of $k-1$ copies of $Y$. Consequently, the dual process is of the form $A^{(k)}_t = (A_1,\dots,A_{k-1})_t$. A good reference for an analytical description of duality for interacting particle systems is L\'opez and Sanz \cite{LS00}. Here, the interpretation is slightly different from the case $k=2$. In the latter, $A \in Y$ represents the set of sites at which there were particles, while $A^{(k)} = (A_1,\dots,A_{k-1}) \in Y^{(k)}$ represents the positions $A_i$ of the particles that trace back to opinion $i$ (this will become clear in the graphical representation), possibly with $A_i = \emptyset$ for some $i \in \{1,\dots,k-1\}$. The dual function $H\colon\,X \times Y^{(k)} \to \mathbb{R}$ now reads
\[
H(\eta,A^{(k)}) = 1_{\left\{\eta(x)=i\,\, \forall\, x\in A_i \,\, \forall\, i\in\{1,\dots,k-1\}\right\}},
\]
and the duality relation becomes
\[
\P^\eta\big(\eta_t(x)=i \,\, \forall\, x\in A_i \,\,\forall\, i\in\{1,\dots,k-1\}\big) 
= \P^{A^{(k)}}\big(\eta(x_t)=i \,\,\forall\, x_t\in (A_i)_t\,\, \forall\, i\in\{1,\dots,k-1\} \big)
\]
for every initial configuration $\eta\in X$ and every initial state $A^{(k)} \in Y^{(k)}$. 

Let us consider the case in which $k=n$, and take $A^{(k)}$ such that $A_i = \{i\}$ for all sites, i.e., we place a continuous-time random walk at every site. Fix $\eta = \eta_0 \in X$ to be the configuration in which each site has its personal opinion, say vertex $i\in[n]$ has opinion $i$, for which $|[n]|=|W|$. If one considers the configuration $\hat{\eta}_t$ defined as 
\[
\hat{\eta}_t(i) = \eta_0((A_i)_t), \qquad i\in [n],\, t>0,
\]
then the duality relation says that $\hat{\eta}_t$ has the same distribution as the state $\eta_t$ of the voter model at time $t$ with initial configuration $\eta_0$. Define the \emph{consensus time} as
\[
\tau_{\rm cons}=\inf\{t\geq 0\colon\, \eta_t(x) = \eta_t(y)\,\, \forall\, x,y\in V\}
\] 
and the \emph{coalescence time} as
\[
\tau_{\rm coal,n} = \inf\{t \geq 0\colon\, \text{ all $n$ particles have coalesced into one}\}.
\]
In particular, it follows that the consensus time has the same law of the coalescence time, i.e., 
\[
\E[\tau_{\rm cons}] = \E[\tau_{{\rm coal},n}].
\]
It follows that if we consider the same model with $2 \leq k < n$ opinions and a given initial configuration $\eta_0$ of them among $[n]$, then the distribution of $\tau_{\rm cons,k}$ will be stochastically dominated by the distribution of $\tau_{{\rm coal},n}$ because $\{\tau_{\rm cons,k}>t\} \supseteq \{\tau_{\rm cons}>t\}$, so $\P^{\eta_0}(\tau_{\rm cons,k}\leq t) \leq \P(\tau_{\rm cons} \leq t)$ for all $t\geq 0$. (With $\tau_{\rm cons,k}$ we mean the consensus time of the voter model with $k$ opinions, with the convention that $\tau_{\rm cons,n} = \tau_{\rm cons}$.) Thus,
\[
\E^{\eta_0}[\tau_{\rm cons,k}] \leq \E[\tau_{{\rm coal},n}], \qquad 2 \leq k < n.
\]
Moreover, it can be proved that, in the case where $k=2$ and the initial distribution $\mu_u$ is given by the product measure of parameter $u \in (0,1)$ Bernoulli random variables,
\[
2u(1-u)\, \E[\tau_{{\rm coal},n}] \leq \E^{\mu_u}[\tau_{\rm cons,2}] \leq \E[\tau_{{\rm coal},n}], 
\qquad u \in (0,1).
\]


\paragraph{Graphical representation of the dual process.}

We conclude this section by giving the duality principle using the graphical representation. Start with the same setting as above: voter model with state space $X=W^{[n]}$, $|W|=k$, defined by its generator and its rates. Consider the graph $\{(j,t)\colon\, j \in [n], t \geq 0\}$ and independent rate-1 Poisson processes $(N_i(t))_{t \geq 0}$, $i\in[n]$. The dynamics is the following: if $\bar{t}$ is an event of the clock $N_i$ for some $i \in [n]$, then draw an arrow from $(\bar{t},j)$ to $(\bar{t},i)$, where $j \in [n]$ is chosen with probability $p(i,j)$. These transition probabilities coincide with the ones given above. In other words, an event represented by an arrow $j \to i$ means that at time $\bar{t}$ the voter at site $i$ decides to adopt the opinion of the voter at site $j$. 

Given any initial configuration $\eta_0 \in X$, we let the opinions flow upwards, starting at time $t=0$, and any time they encounter the base of an arrow they follow its direction, changing the opinion that is at the tip of the arrow. In the case of two-opinions ($0$ and $1$), this construction can be seen as a percolation process where a fluid is placed at $t = 0$ in the 1-sites of $\eta_0$ and flowing up the structure: the arrows are the pipes and the tips are the dams (see Durrett \cite{D88}). An example of this graphical representation is given in Figure \ref{fig:Voter_model_8-cycle} for the 8-cycle graph with $k = n$.

\begin{figure}[htbp]
\hspace{-1cm}
\includegraphics[width=1\linewidth]{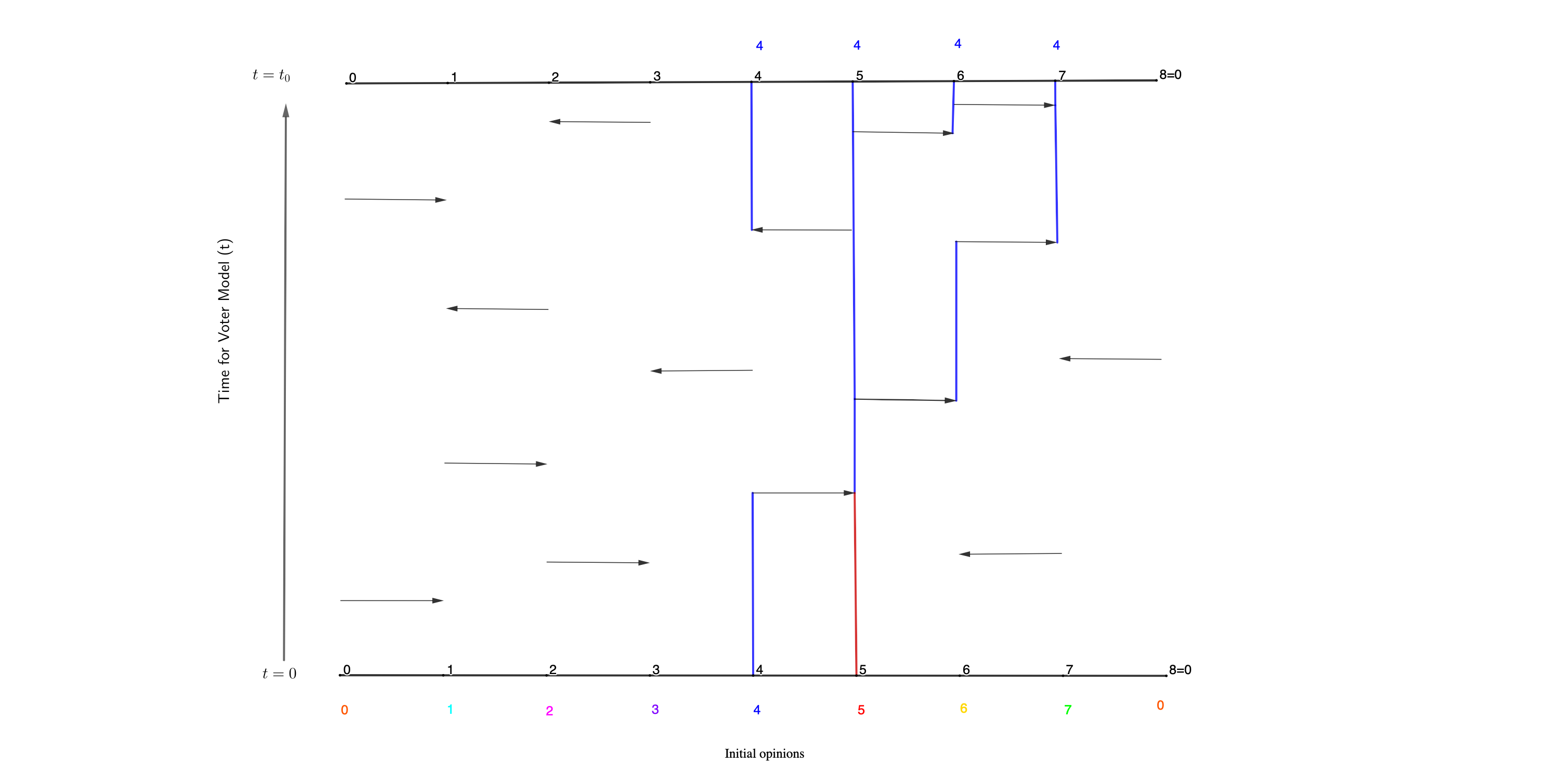}
\vspace{-0.5cm}
\caption{\small Realisation of the voter model on the 8-cycle graph where the initial configuration is given by a different opinion for each site: site $i \in [n]$ has opinion $i$. Times flows vertically up to a finite horizon $t_0 > 0$, and the black horizontal arrows represent the events of the independent Poisson processes $(N_i(t))_{t \geq 0}$ described above. In this picture we highlight the evolution of opinion 4 that at time $t_0$ is shared among the sites $\{4,5,6,7\}$.}
\label{fig:Voter_model_8-cycle}
\end{figure}

Let us now fix a time horizon $t_0>0$ and position a walk in $(i,t_0)$ for all sites $i \in [n]$. We let these walks evolve independently as follows: they move downwards through the graph $\{(j,t_0-t)\colon\, j \in [n], t \in [0,t_0]\}$, and any time they encounter the tip of an arrow they follow it in the opposite direction. Furthermore, if one of them moves to a site already occupied by another walk, then the two walks coalesce into a single (independent) one. Alternatively, the process can be described as follows: each of the walks waits an exponential time of parameter $1$ and, given the current position $x \in [n]$, moves to $y \in [n]$ with probability $p(x,y)$. The same for the coalescing condition, i.e., each time at least two walks meet at the same site they coalesce into a single partcile. Denote now by $A^{(k),t_0}_t = (A_1^{t_0},\dots,A_n^{t_0})_t$ the resulting system of $n$ coalescing random walks (CRWs) evolving as above, where, for each $i \in [n]$, $(A_i)^{t_0}_t$ is the position of the walk starting in $(i,t_0)$ at time $t$, in particular, $A^{(k),t_0}_0 = [n]$. Following the previous example, Figure \ref{fig:CRW_8-cycle} gives a realisation of the CRWs system in the same 8-cycle graph shown in Figure \ref{fig:Voter_model_8-cycle}.

\begin{figure}[htbp]
\begin{center}
\includegraphics[width=1\linewidth]{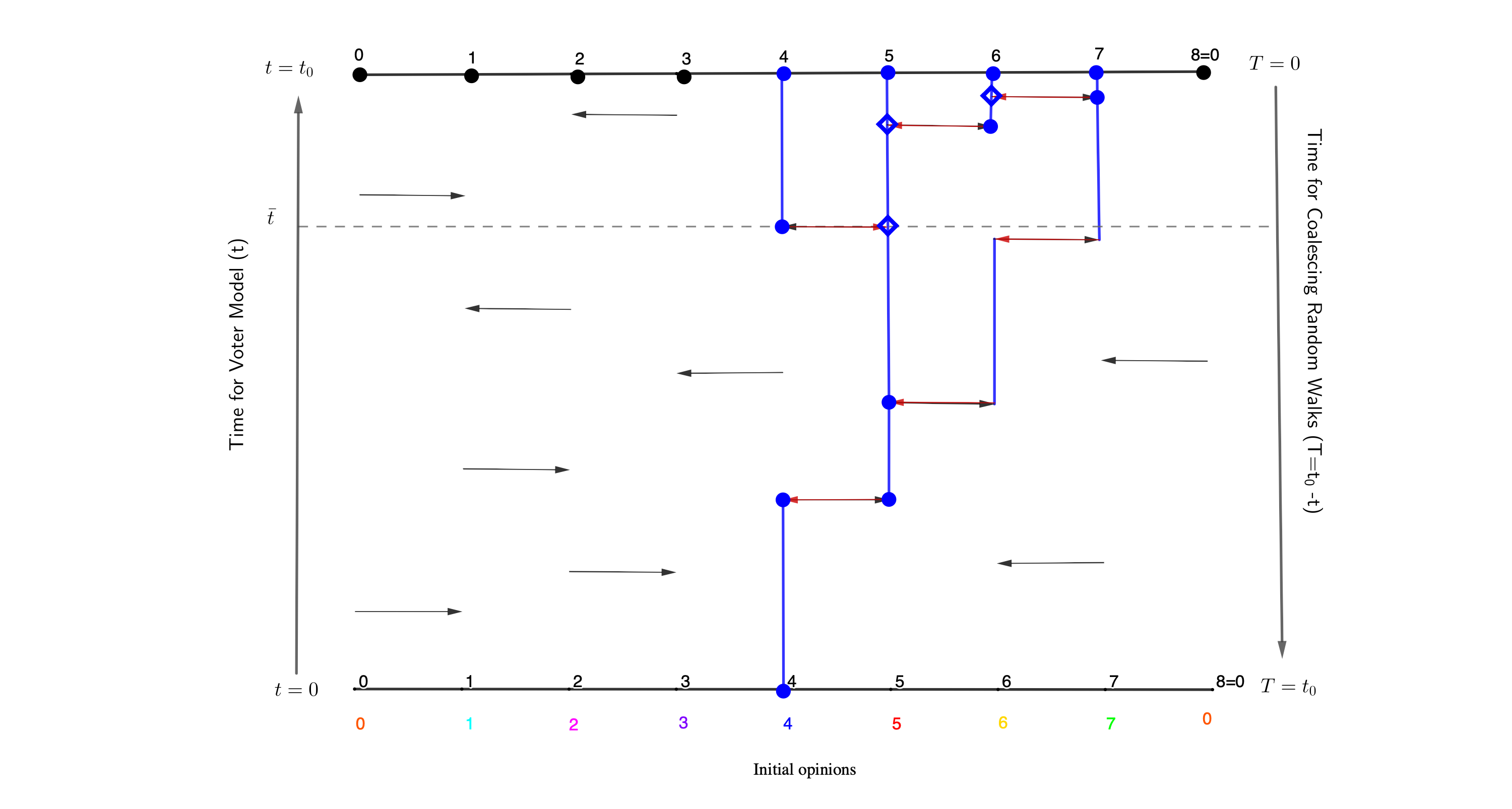}
\end{center}
\vspace{-0.5cm}
\caption{\small Realisation of a coalescing random walk system, starting at $t = t_0$, as dual process of the voter model in Figure \ref{fig:Voter_model_8-cycle}. In the figure the large dots in the graph at $t_0$ represent the initial positions of the walkers, the blue dots are the ones for which we derived the whole trajectory up to time $t = 0$, and coincide with the sites in which we had a local consensus with respect to opinion 4 in the voter model. The blue dots in all the other positions describe the evolution of the position of the (blue) walks, while the blue diamonds $\textcolor{blue}{\diamond}$ mean that in such a time-space location a clustering occurred, i.e., two walks coalesced into one. Moreover, the red horizontal arrows shows the direction to be followed by the walks, which are the opposite w.r.t.\ the direction of the original (black) arrows of the voter model. Note that by time $\bar{t}$ all walks started in $\{4,5,6,7\}$ coalesced into one, meaning that the opinion of each vertex in the latter set at time $t = t_0$ is the same of the one held by vertex 4 at time $t = 0$.}
\label{fig:CRW_8-cycle}
\end{figure}

Given this construction it is immediate to see that the opinion held by a vertex $i$ at time $t_0$ can be derived by \emph{tracing back in time} the path of the random walkers up to time $t = 0$. Thus,
\[
\eta_{t_0}(i) = \eta_0((A_i)_{t_0}^{t_0}) \qquad \forall\, i \in [n],\,\forall\, t_0 > 0.
\]
With the latter we derived the same result obtained using the duality relation. All the other results regarding the equivalence in distribution between the coalescing time and the consensus time with $k=n$ opinions, i.e. $\eta_0=[n]$, follow directly. 


\subsection{VM on the complete graph}

As a prelude we look at the VM on the complete graph, for which computations can be carried through explicitly. Indeed, the number of $1$-opinions at time $t$, given by 
\[
O^N_t = \sum_{i \in V} \xi_t(i),
\]
performs a continuous-time nearest-neighbour random walk on the set $\{0,\ldots,N\}$ with transition rates
\[
\begin{array}{lll}
&n \to n+1 &\text{at rate } n(N-n)\frac{1}{N-1},\\[0.4cm] 
&n \to n-1 &\text{at rate } (N-n)n\frac{1}{N-1}.
\end{array}
\]
This is the same as the Moran model from population genetics. Put $\cO^N_t=\frac{1}{N} O^N_t$ for the \emph{fraction of $1$-opinions at time $t$}. The following is a well-known fact. 

\begin{lemma}
\label{lem:3.1}
The process
\[
(\cO^N_{sN})_{s \geq 0}
\]
converges in law as $N\to\infty$ to the \emph{Fisher-Wright diffusion} $(\chi_s)_{s \geq 0}$ on $[0,1]$ given by
\[
\dd \chi_s = \sqrt{2\chi_s (1-\chi_s)}\,\dd W_s,
\]
where $(W_s)_{s \geq 0}$ is standard Brownian motion.
\end{lemma} 

The number of discordant edges equals
$$
D^N_t = \frac{O^N_t(N-O^N_t)}{2}.
$$
Recall that $\cD^N_t=\frac{1}{M}D^N_t$ denotes the \emph{fraction of discordant edges at time $t$}, with $M=\binom{N}{2}$ for the complete graph. Since
\[
\cD^N_t = \frac{O^N_t(N-O^N_t)}{N(N-1)} = \frac{N}{N-1}\,\cO^N_t(1-\cO^N_t),
\]
it follows that
\[
(\cD^N_{sN})_{s \geq 0}
\]
converges in law as $N\to\infty$ to the process 
$$
\big(\chi_s(1-\chi_s)\big)_{s \geq 0}.
$$
In the mean-field setting of the complete graph, the fraction of discordant edges is the \emph{product} of the fractions of the two opinions. The latter property \emph{fails} on non-complete graphs, in particular, on random graphs.

\medskip
We close this section with a proof of Lemma~\ref{lem:3.1}. The proof is instructive because it shows how computations with generators can be useful.

\begin{proof}
It is enough to prove convergence of generators (see Ethier, Kurtz \cite[Chapter 10]{EK86}). The rescaled voter model on the complete graph is the birth-death process with state space $\{0,\frac{1}{N},\dots,1-\frac{1}{N},1\}$ and infinitesimal generator $\hat{L}_N$ given by
\[
(\hat{L}_N f)\left(\tfrac{i}{N}\right) =  N\,\tfrac{i}{N}\,(N-1) \big[f\left(\tfrac{i+1}{N}\right) 
+ f\left(\tfrac{i-1}{N}\right) - 2f\left(\tfrac{i}{N}\right)\big], \quad i\in\{0,\dots,N\}.
\]
This can be seen by conditioning on the number of steps by the non-rescaled process and using the definition of generator. We have to prove that $\hat{L}_N$ converges as $N\to\infty$ to the generator of the FW-diffusion given by
\[
(Lf)(y) = y(1-y)\,f''(y), \quad y \in [0,1],
\]
i.e., 
\[
\lim_{N\to\infty} (\hat{L}_N f)\left(\tfrac{i}{N}\right) = (Lf)(y) \quad \text{ when } \quad
\lim_{N\to \infty} \frac{i}{N} = \lim_{N\to \infty} \frac{i_N}{N} = y.
\]
We have to specify which set of test functions $f$ we consider. Let $C([0,1])$ be the set of $\R$-valued continuous functions on the unit interval, and define
\[
C_0([0, 1]) = \{f \in C([0, 1])\colon\, f(0) = f(1) = 0\}.
\]
Since the `local speed' of the FW-diffusion is given by the diffusion function $g_{FW}\colon\, [0,1] \to [0, \infty)$ with $g_{FW}(y) = y(1 - y)$, it is possible to show that the domain $D(L)$ of the generator $L$ of the FW-diffusion is a subset of $C_0([0,1])$. For general Markov processes it is not easy to characterise $D(L)$, but it suffices to consider the action of $L$ on a subset $K(L) \subset D(L)$ that is large enough to maintain the generality of the argument. We refer to $K(L)$  as a core of the generator and require that $K(L)$ is dense in $C_0([0,1])$ with respect to the supremum norm. Then the required generality is obtained via continuous extension. In our case it suffices to choose
\[
K(L) = \{f \in C_0([0,1])\colon\, f \text{ is infinitely differentiable}\}.
\]
This choice enables us to use the Taylor expansion of any test function up to any order. Using the Taylor expansion of $f$ around $i/N$ up to second order in the generator, we get
\[
\begin{aligned}
&f\left(\tfrac{i+1}{N}\right) + f\left(\tfrac{i-1}{N}\right) - 2f\left(\tfrac{i}{N}\right)\\[0.2cm] 
&= \big(\tfrac{i+1}{N}-\tfrac{i}{N}\big) f'\left(\tfrac{i}{N}\right)
+ \big(\tfrac{i-1}{N}-\tfrac{i}{N}\big) f'\left(\tfrac{i}{N}\right) 
+ N^{-2} f''\left(\tfrac{i}{N}\right) + \mathcal{O}(N^{-3})\\[0.2cm]
&= N^{-2} f''\left(\tfrac{i}{N}\right) + \mathcal{O}(N^{-3}).
\end{aligned}
\]
Therefore
\[
(\hat{L}_N f)\left(\tfrac{i}{N}\right) =  N^2\,\tfrac{i}{N}\,\left(1-\tfrac{i}{N}\right) 
\big[N^{-2} f''\left(\tfrac{i}{N}\right) + \mathcal{O}(N^{-3})\big], \quad i\in\{0,\dots,N\}.
\]
Hence, when $\lim_{N\to \infty} \tfrac{i}{N} = \lim_{N\to \infty} \tfrac{i_N}{N} = y$, thanks to the continuity of $f\in K(L)$, we get
\[
\lim_{N\to\infty} (\hat{L}_N f)\left(\tfrac{i}{N}\right) = (Lf)(y).
\]
\end{proof}


\subsection{VM on the random regular graph}

Consider the regular random graph $G_{d,N} = (V,E)$ of degree $d \geq 3$, consisting of 
$$ 
\begin{aligned}
&|V| = N \text{ vertices},\\[0.2cm] 
&|E| = M = \frac{dN}{2} \text{ edges}.
\end{aligned}
$$
Denote the law of $G_{d,N}$ by $\PP$. Chen, Choi, Cox \cite{CCC16} consider the fraction of $1$-opinions at time $t$,
\[
\cO^{N}_t = \frac{1}{N} \sum_{i \in V} \xi_t(i),
\]
and show that
\[
(\cO_{sN}^{N})_{s \geq 0}
\]
converges in law as $N\to\infty$ to the Fisher-Wright diffusion $(\chi_s)_{s \geq 0}$ given by
\[
\dd \chi_s = \sqrt{2\theta_d\chi_s (1-\chi_s)}\,\dd W_s,
\]
where $(W_s)_{s \geq 0}$ is standard Brownian motion, and 
\[
\theta_d = \frac{d-2}{d-1}.
\]
plays the role of a \emph{diffusion constant}.


\paragraph{Main theorems.}

For $u \in (0,1)$, let ${\bf P}_u$ be the law of $(\cD^N_t)_{t \geq 0}$ starting from $[{\rm Bernoulli}(u)]^N$.
For $\eta \in \{0,1\}^V$, let ${\bf P}$ be the law of $(\cD^N_t)_{t \geq 0}$ starting from $\eta$. 

\begin{theorem}
\label{thm:3.2} 
Fix $u\in(0,1)$. Then, for any $t_N \in [0,\infty)$,
\[
\left| {\bf E}_u\left[\cD^N_{t_N}\right] - 2u(1-u)\,f_d(t_N)\,\mathrm{e}^{-2 \theta_d \frac{t_N}{N}}\right| 
\overset{\mathbb{P}}{\longrightarrow} 0, 
\]
where  
\[
f_d(t) = {\bf P}^{\cT_d}(\tau_{\rm meet} > t),
\]
with ${\bf P}^{\cT_d}$ the law of two independent random walks on the infinite $d$-regular tree $\mathcal{T}_d$ starting from the endpoints of an edge.
\end{theorem}

The profile function $f_d$ is given by
\[
f_d(t) = \sum_{k=0}^{\infty}
\ee^{-2t}\frac{(2t)^k}{k!}\\
\sum_{l>\lfloor\frac{k-1}{2}\rfloor}\binom{2l}{l}\frac{1}{l+1} 
\Big(\frac{1}{d} \Big)^{l+1}\Big(\frac{d-1}{d}\Big)^{l},
\vspace{0.3cm}
\]
and satisfies $f_d(0) = 1$ and $f_d(\infty) = \theta_d$. Note that Theorem \ref{thm:3.2} shows three times scales (see Figures~\ref{fig:sim}--\ref{fig:scatter}): 

\medskip\noindent
\begin{tabular}{lll}
&short: &$t_N \ll N$,\\
&moderate: &$t_N \asymp N$,\\ 
&long: &$t_N \gg N$.
\end{tabular}

\vspace{0.2cm}
\begin{figure}[htbp]
\begin{center}
\includegraphics[width=5.5cm]{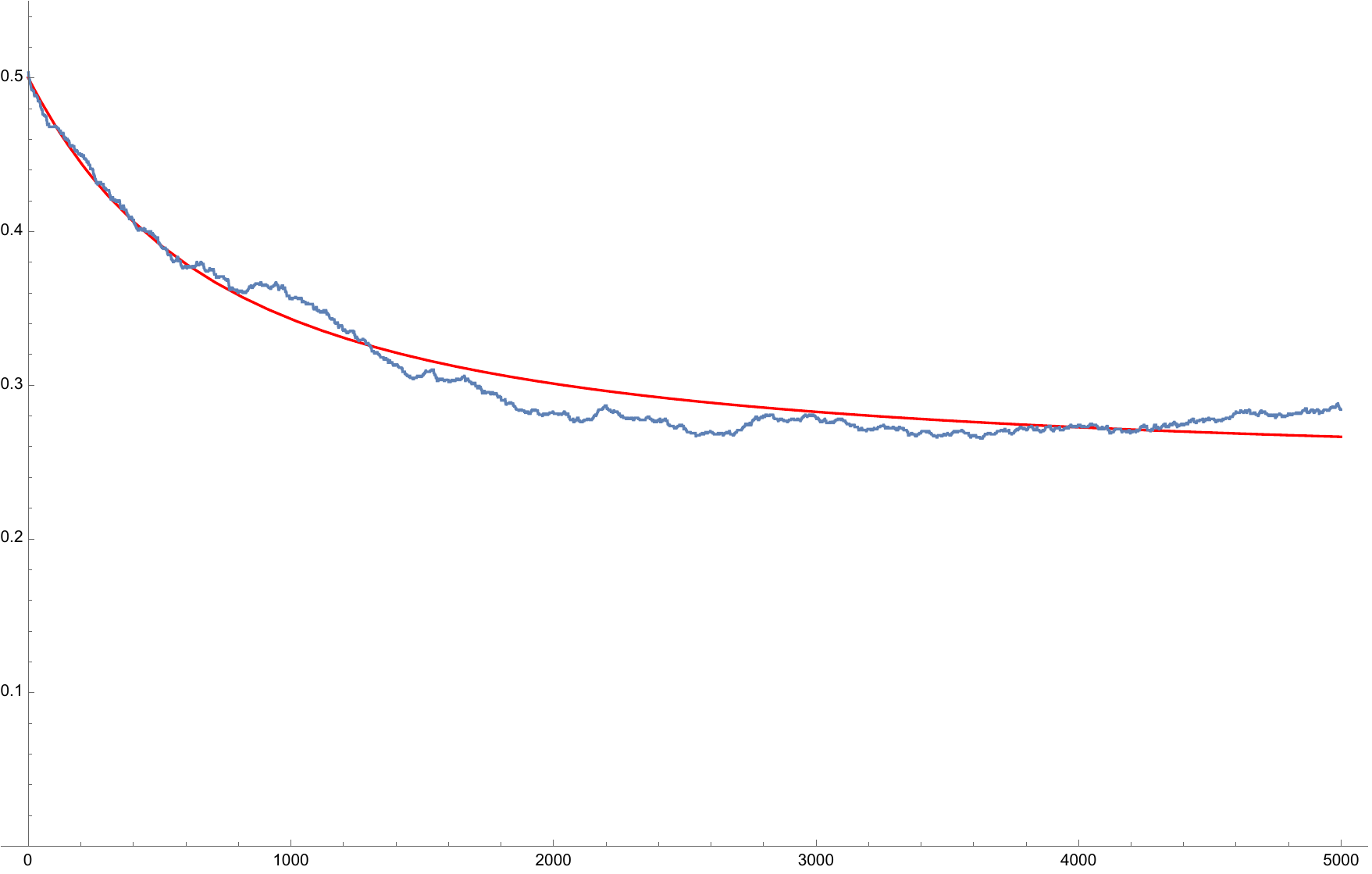}
\hspace{1cm}
\includegraphics[width=5.5cm]{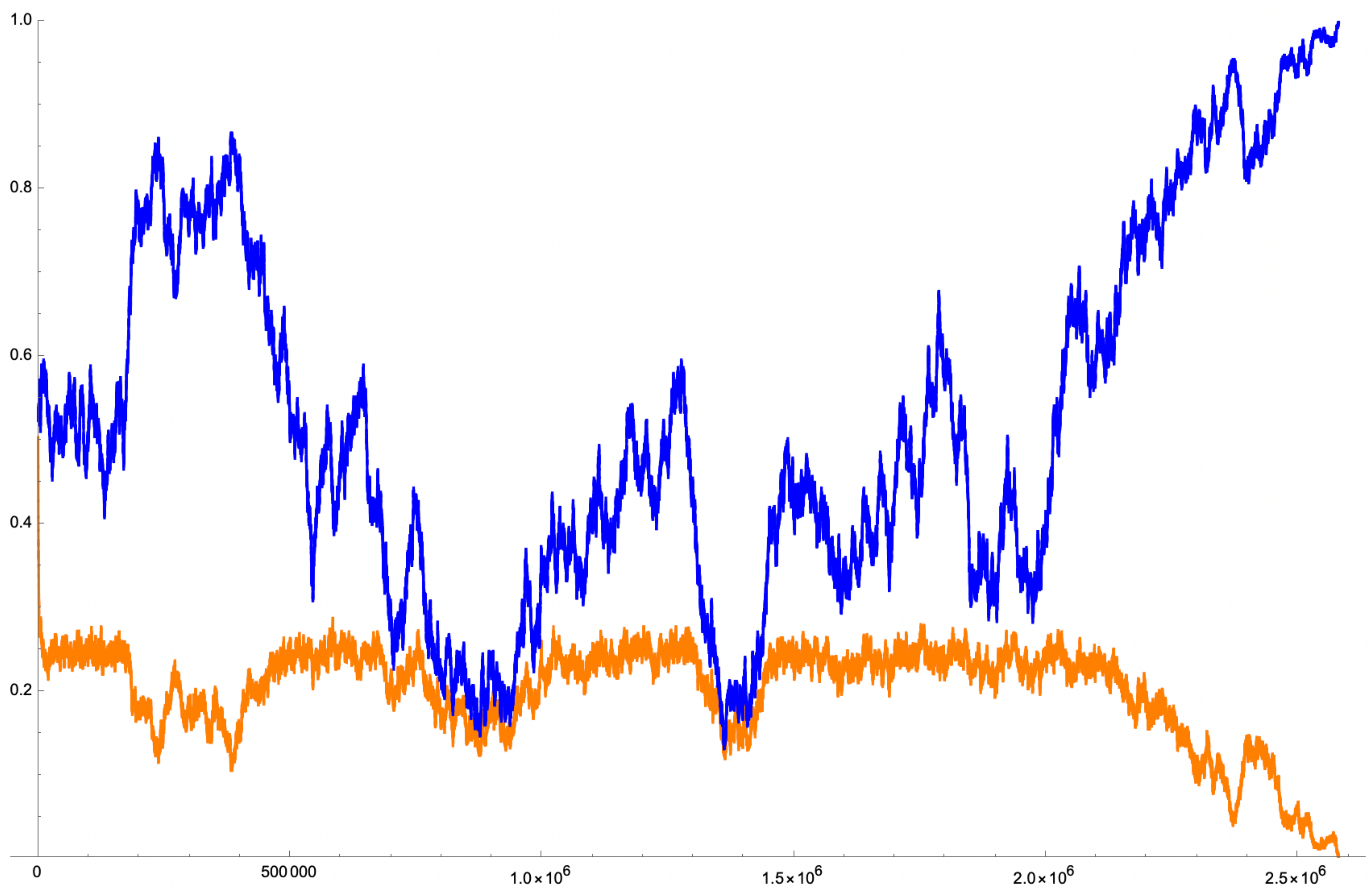}
\end{center}
\vspace{-0.2cm}
\caption{\small A single simulation for $N=1000$, $d=3$, $u=0.5$. \emph{Left}: In blue the fraction of discordant edges up to $t=5$, in red the function $t \mapsto 2u(1-u)\,f_d(t)$. \emph{Right}: In blue the fraction of $1$-opinions up to consensus,
in orange the fraction of discordant edges up to consensus.}
\label{fig:sim}
\end{figure}

\vspace{0.2cm}
\begin{figure}[htbp]
\begin{center}
\includegraphics[width=5cm]{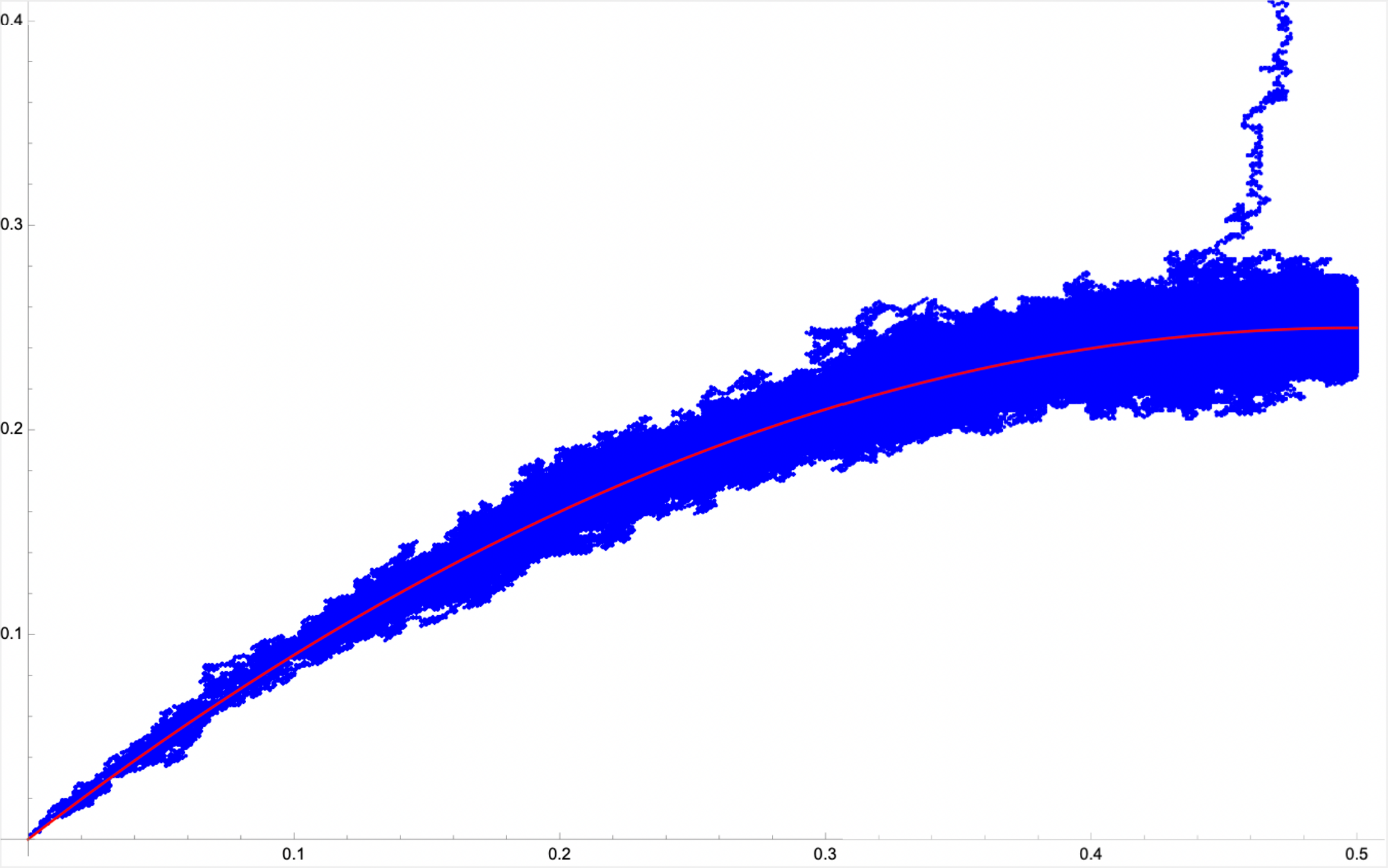}
\end{center}
\vspace{-0.2cm}
\caption{\small Scatter plot for the same simulation: the fraction of discordant edges versus the fraction of the minority opinion. The piece sticking out corresponds to short times. The curve in red is $x \mapsto x(1-x)$, which says that the fraction of discordant edges is close to the product of the fractions of the two opinions.}
\label{fig:scatter}
\end{figure}

\begin{theorem}
\label{thm:3.3}
The following concentration properties hold:
\begin{enumerate}
\item[(i)] 
Let $t_N$ be such that $t_N/N \to 0$. Then, for every $\varepsilon>0$,
\[
\sup_{\eta\in\{0,1\}^V} \Prob_\eta\left(\left|\cD^N_{t_N}-\Expect_\eta[\cD_{t_N}^N] \right|>\varepsilon \right)
\overset{\PP}{\longrightarrow} 0.
\]
\item[(ii)] 
Let $t_N$ be such that $t_N/N \to s \in (0,\infty)$. Then, for every $u\in(0,1)$,
\[
\sup_{x\in[0,1]} \left| \Prob_u\left(\cD^N_{t_N} \leq x \right) - \Prob_u\left(\chi_s(1-\chi_s)\le x \right) \right| 
\overset{\PP}{\longrightarrow} 0.
\]
\end{enumerate}
\end{theorem}

\begin{theorem}
\label{thm:3.4}
Fix $u\in(0,1)$. Then, for every $\delta,\epsilon>0$,
\[
\Prob_u \Big( \sup_{0 \leq t \leq N^{1-\delta}}\big|\cD_t^N- \Expect_u[\cD_t^N]\big|>\varepsilon \Big) 
\overset{\PP}{\longrightarrow} 0.
\]
\end{theorem}

The proofs of Theorems~\ref{thm:3.2}--\ref{thm:3.4} are based on the classical notion of \emph{duality} between the voter model and a collection of coalescent random walks. A crucial role is played by properties of coalescing random walks that hold in \emph{mean-field geometries}. In particular, in Oliveira \cite{O13} it is shown that
\[
\lim_{N\to\infty}\frac{\Expect[\tau_{\rm coal}]}{\Expect[\tau^{\pi\otimes\pi}_{\rm meet}]} = 2,
\]
where $\tau_{\rm coal}$ is the coalescence time of $N$ random walks each starting from a different vertex, $\Expect$ is expectation w.r.t.\ these random walks, while $\tau_{\rm meet}^{\pi\otimes\pi}$ is the meeting time of two random walks independently starting from the stationary distribution $\pi$. A further discussion of this result is postponed to Section~\ref{s.coalmeet}.

\medskip\noindent
1. On time scales $o(\log N)$, below the typical distance between two vertices, the analysis proceeds by coupling two random walks on the $d$-regular random graph with two random walks on the $d$-regular tree, both starting from adjacent vertices. Because the tree is regular, the distance of the two random walks can be viewed as the distance to the origin of a single biased random walk on $\N_0$ starting from $1$. Note that the same does not hold when the tree is not regular.

\medskip\noindent
2. On time scale $\Theta(\log N)$, the scale of the typical distance between two vertices, the coupling argument is combined with a finer control of the two random walks on the $d$-regular random graph.

\begin{lemma}
\label{lem:3.5}
There is a sequence of random variables $(\theta_{d,N})_{N\in\N}$ converging to $\theta_d$ such that
\[
\lim_{N\to\infty} \sup_{t \geq 0} 
\left|\,
\frac{\Prob(\tau_{\rm meet}^{\pi\otimes\pi}>t)}{\exp[-2\theta_{d,N}(t/N)]} -1
\,\right| = 0
\quad \text{in probability}.
\]
\end{lemma}

\noindent
Lemma \ref{lem:3.5}, together with a first-moment argument, is enough to compute the evolution of the expected number of discordant edges on every time scale. 

\medskip\noindent
3. In order to obtain concentration, a much deeper analysis is required. Roughly, in order to have proper control on the correlations between the discordant edges, we must analyse a dual system of random walks whose number grows with $N$. An upper bound is derived for the number of meetings of a poly-logarithmic number of independent random walks evolving on the random graph for a time $N^{1-o(1)}$. This is exploited to derive an upper bound for the deviation from the mean that is exponentially small in $N$ and uniform in time. This upper bound can be translated into a concentration estimate by taking a union bound.

\medskip\noindent
OPEN PROBLEMS:
\begin{itemize}
\item
We expect that Theorems \ref{thm:3.2}--\ref{thm:3.3} can be extended to non-regular sparse random graphs. We do not have a conjecture on how the function $f_d$ and the diffusion constant $\theta_d$ modify in this more general setting.
\item
We expect that Theorem \ref{thm:3.3} can be strengthened to the statement that, for every $u \in (0,1)$, every $t_N$ such that $t_N/N \to 0$ and every $C_N \to \infty$,
\[
\Prob_u\big(\left|\cD^n_{t_N}-\Expect_u[\cD^N_{t_N}]\right| > C_N \sqrt{t_N/N}\,\big)\overset{\PP}{\longrightarrow} 0.
\]
\end{itemize}


\paragraph{Directed graphs.} 

For directed sparse random graphs more can be said. The setting is the \emph{directed configuration model} with prescribed in-degees $d^{\mathrm{in}} = (d_i^{\mathrm{in}})_{i=1}^N$ and out-degees $d^{\mathrm{out}} = (d_i^{\mathrm{out}})_{i=1}^N$ in which directed half-edges are matched randomly. 

\begin{theorem}
\label{thm:3.6}
{\rm (Avena, Capannoli, Hazra, Quattropani \cite{ACHQ23}, Capannoli \cite{C24pr}).}\\ 
Under mild conditions on the in-degrees $d^{\mathrm{in}}$ and the out-degrees $d^{\mathrm{out}}$, the same scaling applies and an explicit formula can be derived both for the profile function $f_{d^{\mathrm{in}},d^{\mathrm{out}}}$ and for the diffusion constant $\theta_{d^{\mathrm{in}},d^{\mathrm{out}}}$.
\end{theorem}

\noindent
For instance, if $d^{\mathrm{in}} = d^{\mathrm{out}}$ (= Eulerian graph), then 
\[
\theta_{d^{\mathrm{in}},d^{\mathrm{out}}} = \Big(\tfrac{m_2}{m_1^2}-1+\sqrt{1-\tfrac{1}{m_1}}\,\Big)^{-1}
\] 
with $m_1,m_2$ the first and the second moment of the limit of the empirical degree distribution.

From the fact that $\theta_{d^{\mathrm{in}},d^{\mathrm{out}}}$ is an explicit function of $d^{\mathrm{in}}$ and $d^{\mathrm{out}}$, it is possible to analyse its behaviour as a function of both. Relevant questions are: Does the consensus time speed up or slow down when, in the setting of constant out-degrees (or in-degrees), the variability of the in-degrees (or out-degrees) is increased? What is the effect of positive or negative correlation between the in-degrees and out-degrees? If certain features of the degree sequences are constrained, then what are the guiding principles to minimise or maximise the consensus time?


\subsection{VM meeting times and coalescence times}
\label{s.coalmeet}


\paragraph{Discordance and meeting times.}

As show in Section~\ref{ss.dual}, the graphical representation allows us to write the probability that a fixed edge is discordant at time $t$ in terms of the meeting time of random walks. Let $D_t=\{e = (x,y) \in E\colon\, \eta_t(x) \neq \eta_t(y)\}$ be the set of discordant edges at time $t$ for the voter model with initial distribution of opinions given by the product measure of independent $\text{Bern}(u)$, $u\in(0,1)$. Moreover, define
\[ 
\tau^{(x,y)}_{\rm meet} = \inf\{t \geq 0\colon\, X_t^x = Y_t^y\}
\]
to be the first meeting time of two independent random walks $X,Y$ having initial position $(X_0,Y_0) = (x,y)$, $x,y\in V$. Then, by duality,
\[
\begin{aligned}
\mathbf{P}_u(e\in D_t) 
&= \mathbf{P}\big(X_s^x\neq Y_s^y\,\, \forall\,0 \leq  s \leq t\big) \times 
\mathbf{P}_u \big(\eta_0(X_t^x) \neq \eta_0(Y_t^y)\big)\\
&= \mathbf{P}\big(\tau^{(x,y)}_{\rm meet} > t\big) \times 2u(1-u).
\end{aligned}
\]
(See Figure~\ref{fig:DvsM}.)

\begin{figure}[htbp]
\begin{center}
\includegraphics[width=0.8\linewidth]{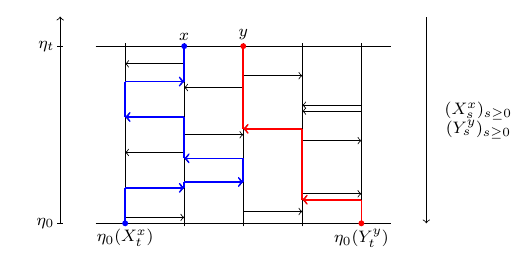}
\end{center}
\vspace{-0.5cm}
\caption{\small Discordance versus meeting.} 
\label{fig:DvsM} 
\end{figure}


\paragraph{Coalescence times.}

Recall the remarks below Theorems~\ref{thm:3.2}--\ref{thm:3.4}. Let $m(G)= \mathbb{E}[\tau^{\pi\times\pi}_{\rm meet}]$ be the expected first meeting time of two independent random walks starting from their stationary distribution $\pi$ over a finite connected graph $G=(V,E)$. Recall that $\tau_{\rm coal}$ is the first time such that a system of $n$ random walks coalesce into a single one, where $n=|V|$. On the complete graph with $n$ vertices it can be proved (see Aldous and Filll \cite[Chapter 14]{AF02}) that 
\[
\frac{\tau_{\rm coal}}{m(G)} = \frac{\tau_{\rm coal}}{(n-1)/2} \overset{d}{=} \sum_{i=2}^n Z_i,
\]
where $(Z_i)_{i \in \N\setminus\{1\}}$ are independent random variables with law
\[ 
Z_i \overset{d}{=} \text{Exp}\Big(\frac{1}{\binom{i}{2}}\Big), \quad i \in \N\setminus\{1\}.
\]
In particular, 
\[
\text{Law}\Bigg(\frac{\tau_{\rm coal}}{m(G)}\Bigg) \underset{n\to\infty}{\longrightarrow}\text{Law}
\Bigg(\sum_{i \in \N\setminus\{1\}} Z_i\Bigg)
\]
and
\[
\frac{\E[\tau_{\rm coal}]}{m(G)} \underset{n\to\infty}{\longrightarrow} 2.
\]

As explained in Oliveira \cite{O13}, $Z_i$ is the time it takes for a system with $i$ random walk to evolve to a system with $i-1$ random walks, rescaled by the expected meeting time of two random walks. For $i = 2$, this is just the (rescaled) meeting time of a pair of random walks, which is an exponential random variable with mean $1$. For $i > 2$, we are looking at the first meeting time among $\binom{i}{2}$ pairs of random walks. It turns out that these pairwise meeting times are independent. Since the minimum of $k$ independent exponential random variables with mean $\mu$ is an exponential random variable with mean $\mu / k$, we deduce that $Z_i$ is exponential with mean $1/\binom{i}{2}$. Such an interpretation comes from the Kingman coalescent \cite{K82}: a pure death process taking values in the collection of partitions of $[n]$ such that if the partition has $j$ sets, then at rate $\binom{j}{2}$ two randomly chosen sets in the partition are joined together.

The aim of Oliveira \cite{O13} was to find a large class of graphs such that the latter equations hold. The class of graphs for which this mean-field-type behaviour had been proved earlier are: the discrete tori $(\Z/m\Z)^d$ with $d\geq 2$ fixed and $m \gg 1$ (see Cox \cite{C89}), and large random $d$-regular graphs (see Cooper, Frieze and Radzik \cite{CFR09}). Oliveira \cite{O13} showed that the equations are satisfied for any finite, transitive and irreducible graph such that the mixing time of a single random walk occurs much faster than $m(G)$. Heuristically, in a short amount of time all but two random walks coalesce, and by that time the two remaining random walks have mixed well and will have to meet from stationarity.


\paragraph{Meeting times on random geometries.}

We aim to compare the meeting times on random regular graphs and inhomogeneous random graphs via the so-called annealed random walk. On the $d$-regular random graph one can exploit the locally-tree-like nature of the environment, which leads one to study the observable on a deterministic $d$-regular tree. Therefore, the problem is reduced to the study of a biased random walk on $\Z$, representing the graph distance of the two random walks, with jump probabilities $1/d$ and $(d-1)/d$. By studying random walks on locally-tree-like random graphs, such as the Configuration Model, the Erd\H{o}s-R\'enyi rando graph and the Chung-Lu random graph (see e.g.\ van der Hofstad \cite{vdH17}, Bollob\'as \cite{B01}), one may be interested in certain observables, such as the first meeting time, of random walks on Galton-Watson trees related to such geometries. 

Fix $n \in \N$, and let $G=([n],E)$ be a locally-tree-like random graph with vertex set $V = [n] = \{1,\dots,n\}$. If we want to study $\tau_{\rm meet}$ on a Galton-Watson tree having a non-trivial offspring distribution, in our case generated by $G$, then the problem becomes much harder, as now $\tau_{\rm meet}$ becomes a \emph{random variable} with respect to the realisation of the tree. For this reason, we cannot rewrite the problem in terms of a single biased walk on $\N_0$. One of the ways to deal with the two very dependent  sources of randomness (the random graph given by the Galton-Watson tree and the stochastic process given by two independent random walks) is to \emph{average out over the environment}. More precisely, let $\mathbf{P}=\mathbf{P}^G$ be the \emph{quenched} law of the two random walks, i.e., the \emph{random measure} that depends on the realisation of the random graph $G$, and let $\P$ be the law of $G$. Let 
\[
\E[\mathbf{P}(\cdot)] = \sum_{G\in\mathfrak{G}_n} \mathbf{P}^\mathcal{G}(\cdot)\,\P(\mathcal{G}=G)
\]
be the \emph{annealed} law of two random walks on $\mathcal{G}$, where $\mathfrak{G}_n$ is the collection of all possible graphs with $n$ vertices. Typically, in this setting, one can attempt to compute the annealed law of $\tau_{\rm meet}$, i.e., $\E[\mathbf{P}(\tau_{\rm meet}\leq t)]$, $t\in\N_0$, and then try to recover the quenched result by a concentration argument. 

Consider the CM $G$ with degree sequence $\mathbf{d}=(d_1,\dots,d_n)$. It consists in a random graph with fixed vertex set $[n]$ and each vertex $x \in [n]$ has a pre-assigned degree $d_x$. We will assume that the degree sequence $\mathbf{d}$ satisfies some regularity conditions (see van der Hofstad \cite{vdH17}). It can be shown that, under these assumptions on the degree sequence, the resulting graph local limit is an unimodular Galton-Watson tree $\mathcal{T}_o$, whose root $o$ has offspring distribution $p = (p_k)_{k \in \mathbb{N}}$ given by 
\[
p_k = \P(D_n=k) = \frac{1}{n} \sum_{x\in[n]} 1_{d_x=k},
\]
where $D_n$ is the degree of an uniformly chosen vertex, while the other sites have offspring distribution given by the size-biased law
\[
 \mu(k) = \frac{(k+1)\,p_{k+1}}{\E[D_n]}, \qquad k \in \N.
\]

Consider now two asynchronous independent random walks $X,Y$ with initial position $o$. We want to compute $\E[\mathbf{P}(\tau_{\rm meet}\leq t)]$. There is one key observation that makes the above computation possible explicitly: $\tau_{\rm meet}$ is a \emph{local} observable depending only on the neighbours of the two walks at each step. This obvious, but crucial, fact brings two consequences:
\begin{enumerate}
\item
We can exploit the \emph{local weak limit} of the CM.
\item 
We can make an explicit description of the annealed law in terms of a non-Markovian process
\[
\P^{\rm ann} (\cdot) = \E[\mathbf{P}(\cdot)]
\]
that at each step simultaneously explores the random graph and lets the walks move. This equivalent description is only based on the fact that in the definition the description of the meeting time is given by a local exploration of the graph structure, not a global exploration. For a precise description, see Bordenave et al.\ \cite{BCS18, BCS19}, Cai et al.\ \cite{CCPQ23} and Avena et al.\ \cite{ACHQ23}. 
\end{enumerate}

Suppose that $t$ is uniformly bounded in $n$, i.e., $t = t_n = \mathcal{O}(1)$ as $n\to \infty$. In order to compute $\P^{\rm an}(\tau_{\rm meet}\leq t)$, we initialise $G$ having $n$ vertices equipped with $\{d_i\}_{i\in[n]}$ half-edges and without any formed edge (empty matching). We set both the walks $X$ and $Y$ such that $X_0=Y_0=o \sim \text{Unif}([n])$. Since the walks are asynchronous, we select which of the two moves first with probability (w.p.) $\frac12$, say $X$ moves first. Then it selects uniformly at random a half-edge $e$ of $X_0 = o \in [n]$ and selects another half-edge $f$ among the unmatched ones, say that $f$ is incident to some vertex $z \in [n]$. Finally, we create the edge $(o,z)$ by matching the half-edges $e$ and $f$, and we move the walk $X$ to $z$, i.e., $X_1 = z$, while $Y_1 = Y_0 = o$. We iterate this procedure until the walks meet at some time $t$, with the difference that if a walk selects a half-edge that is already matched, then we do not sample a half-edge uniformly at random and we just let the walk move to the other end of the edge.

Since $t = \mathcal{O}(1)$, we can perform a local exploration of a Galton-Watson tree rooted at $o\sim\text{Unif}([n])$ with offspring distribution $\mu$ as above, in place of a local exploration of the whole graph $G$. In order to exhibit an explicit example of $\P^{\rm an}(\tau_{\rm meet}\leq t)$, let us try to compute $\P^{\rm an}(\tau_{\rm meet}=4)$. Suppose that the walks are \emph{non-backtracking}.
\begin{itemize}
\item 
One of the walks, say $X$, moves to a neighbour $z_1$ of $X_0=o$. This event happens w.p.\ $\tfrac12$.
\item 
Since the first meeting happens at time $4$, $Y$ cannot follow immediately $X$, and so $X$ must move again as in Step 1 to some vertex $z_2$. This event happens w.p.\ $\tfrac12$.
\item  
Now $Y$ must follow the unique path connecting $X$ to $Y$ for two steps, from $o$ to $z_1$ and from $z_1$ to $z_2$. This happens with probability 
\[
\big(\tfrac{1}{2}\big)^2\,\frac{1}{d_o}\,\frac{1}{d_{z_1}}.
\] 
The offspring distribution of $o$ is $p_k$ as above, while the degree of $z_1$ is distributed according to $\mu$. It follows that 
\[
\P^{\rm an}(\tau_{\rm meet} = 4) = \big(\tfrac{1}{2}\big)^4 \sum_{k \in \N} \frac{1}{k}\,p_k\,
\sum_{k \in \N} \frac{1}{k}\,\mu(k).
\]
\end{itemize}

\noindent
REMARK: It is unclear what to do when the random walk is backtracking, like simple random walk. 


\subsection{VM on the random regular graph with random rewiring}

What happens when the graph itself evolves over time, e.g.\ the edges are \emph{randomly rewired}? 

Suppose that every pair of edges swaps endpoints at rate $\nu/2M$ with $\nu \in (0,\infty)$, where we recall that $M = \tfrac{dN}{2}$ is the number of edges in the $d$-regular graph with $N$ vertices. With this choice of parametrisation, the \emph{rate at which a given edge is involved in a rewiring} converges to $\nu$ as $N\to\infty$. The voter model on this dynamic random graph evolves as before: at rate $1$ opinions are adopted along the edges that are currently present.  

In work in progress we show that Theorems \ref{thm:3.2}--\ref{thm:3.4} carry over with $\theta_d$ replaced by $\theta_{d,\nu}$ given by a continued fraction.

\begin{theorem}
\label{thm:3.6} 
{\rm (Avena, Hazra, den Hollander, Quattropani \cite{ABHdHQ24}).}
Let $\beta_d = \sqrt{d-1}$ and $\rho_d = \tfrac{2}{d}\sqrt{d-1}$. Then
\[
\theta_{d,\nu} = 1-\frac{\Delta_{d,\nu}}{\beta_d}
\]
with
\[
\Delta_{d,\nu} = \frac{1\,|}{|\,\frac{2+\nu}{\rho_d}} - \frac{1\,|}{|\,\frac{2+2\nu}{\rho_d}} 
- \frac{1\,|}{|\,\frac{2+3\nu}{\rho_d}} - \dots
\]
\end{theorem}

\medskip\noindent
1. Since the $d$-regular random graph locally looks like a $d$-regular tree, the proof proceeds by analysing the meeting time of two random walks on a $d$-regular tree. On short to modest time scales the two random walks do not notice the difference. Work is needed to show that on longer time scales the approximation is still good. 

\medskip\noindent
2. We replace rewiring of edges on the $d$-regular random graph by disappearance of edges on the $d$-regular tree. This is a good approximation because, as soon as one random walk moves along a rewired edge in the $d$-regular random graph, it is thrown far away from the other random walk and meeting becomes difficult.

\medskip\noindent
KEY OBSERVATION: Because $\nu \mapsto \theta_{d,\nu}$ is strictly increasing, the \emph{dynamics speeds up consensus}.


\section{LECTURE 4: The Contact Process (CP)} 
\label{sec:L4}


\subsection{CP on graphs}

Given a connected graph $G=(V,E)$, the contact process is the Markov process $(\xi_t)_{t \geq 0}$ on state space $\{0,1\}^V$ where each vertex is either healthy $(0)$ or infected $(1)$. Each infected vertex becomes healthy at rate $1$, independently of the state of the other vertices, while each healthy vertex becomes infected at rate $\lambda$ times the number of infected neighbours, with $\lambda \in (0,\infty)$ the infection rate. 

The configuration at time $t$ is $\xi_t = \{\xi_t(i)\colon\,i\in V\}$, with $\xi_t(i)$ the state at time $t$ of vertex $i$. In what follows we will analyse the behaviour of the CP on various classes of graphs, both random and deterministic. Our focus will be on understanding how the extinction time
\[ 
\tau_{[0]_N} = \inf\{t \geq 0\,\colon\, \xi_t(i) = 0\,\, \forall\, i \in V\}
\]
behaves as $|V| = N\to\infty$, depending on the value of $\lambda$ and the properties of the graph. We will mostly zoom in on 
\[
\E_{\tau_{[1]_N}}(\tau_{[0]_N}),
\]
the average extinction time starting from the configuration where every vertex is infected. We will see that it is hard to get control on this quantity, so we will have to content ourselves with rough bounds.  


\subsection{CP on the complete graph}

As a prelude we look at the CP on the complete graph, for which computations can be carried through explicitly. Indeed, the number of infections at time $t$, given by
\[
I^N_t = \sum_{i \in V} \xi_t(i),
\]
evolves as a continuous-time nearest-neighbour random walk on the set $\{0,\ldots,N\}$ with transition rates
\[
\begin{array}{lll}
&n \to n+1 &\text{at rate } \lambda n(N-n),\\ 
&n \to n-1 &\text{at rate } n.
\end{array}
\]
Put $\cI^N_t=\frac{1}{N}I^N_t$ for the fraction of infected vertices at time $t$. This process is a continuous-time nearest-neighbour random walk on the set
\[
\left\{0,\tfrac{1}{N},\ldots,\tfrac{N-1}{N},1\right\}
\]
with transition rates
\[
\begin{array}{lll}
&x \to x + N^{-1} &\text{at rate } \lambda x(1-x)N^2,\\ 
&x \to x - N^{-1} &\text{at rate } xN.
\end{array}
\]
This process has a strong drift upward, which becomes zero for $x =1-\tfrac{1}{\lambda N}$, i.e., very close to full infection when $\lambda N \gg 1$.

\begin{theorem}
\label{thm:4.1}
Let $\tau_{[0]_N} = \inf\{t \geq 0\colon\,\xi_t = [0]_N\}$ be the extinction time. Then
\[
\log \E_{[1]_N}(\tau_{[0]_N}) = N [1+ \log (\lambda N)] + o(N), \quad N \to \infty.
\]
\end{theorem}
Thus, the CP on the complete graph is supercritical for all $\lambda>0$ as $N\to\infty$.

\begin{theorem}
\label{thm:4.2}
{\rm (Schapira, Valesin \cite{SV21}.)}\\
For every $\lambda \in (0,\infty)$,
\[
\lim_{N\to\infty} \P_{[1]_N}\left(\frac{\tau_{[0]_N}}{\E_{[1]_N}(\tau_{[0]_N})} > t\right) 
= \ee^{-t} \quad \forall\,t>0.
\]
\end{theorem}

In the mean-field setting of the complete graph, the fraction of infected vertices performs a random walk. The latter property \emph{fails} on non-complete graphs, in particular, on random graphs: $(\cI^N_t)_{t \geq 0}$ loses the Markov property. The CP is harder than the VM because it \emph{does not have a tractable dual}. In fact, it is self-dual.


\subsection{CP on the configuration model}

Consider CP on CM with an empirical degree distribution $f_N$ satisfying $\lim_{N\to\infty} \|f_N-f\|_\infty = 0$ with 
\[
f(k) = k^{-\tau+o(1)}, \qquad k \to \infty,
\]
where $\tau \in (1,\infty)$ is the tail exponent.

\begin{theorem}
\label{thm:4.3}
{\rm (Chatterjee, Durrett \cite{CD09}, Mountford, Mourrat, Valesin, Yao \cite{MMVY16}).}\\
 If $\tau \in (2,\infty)$, then for every $\lambda \in (0,\infty)$ the average time to extinction grows exponentially fast with $N$ $\whp$. 
\end{theorem}

Thus, CP on CM with a power law degree distribution is supercritical regardless of the value of $\lambda$. Apparently, hubs easily transmit the infection. 

\begin{theorem}
\label{thm:4.4}
{\rm (Can, Schapira \cite{CS15}).}\\ 
The same is true for $\tau \in (1,2]$, even though local convergence breaks down.
\end{theorem}

For the CP on the CM with a power law degree distribution there is anomalous scaling of the density of infections $\rho(\lambda)$ as $\lambda \downarrow 0$, namely,
\[
\rho(\lambda) \asymp \left\{\begin{array}{ll}
\lambda^{1/(3-\tau)}, &\tau \in (2,\tfrac52],\\[0.3cm]
\lambda^{2\tau-3}\,[\log(1/\lambda)]^{-(\tau-2)}, &\tau \in (\tfrac52,3],\\[0.3cm]
\lambda^{2\tau-3}\,[\log(1/\lambda)]^{-2(\tau-2)}, &\tau \in (3,\infty). 
\end{array}
\right.
\]
(Mountford, Valesin, Yao \cite{MVY13}, Linker, Mitsche, Schapira, Valesin \cite{LMSV21}.) The three regimes reflect different optimal strategies to survive extinction for small infection rates: the infection survives close to hubs. 

Sharp estimates of the extinction time have been obtained for the CP on the CM with i.i.d.\ degrees $(D_i)_{i=1}^N$ taking values in $\N_0$. When $\E[D_1]<\infty$ we expect a strictly positive critical threshold.  

\begin{theorem}
\label{thm:4.5}
{\rm (Cator, Don \cite{CD21}).}\\ 
Suppose that $\E[D_1]<\infty$ and $\E[2^{-D_1/2}]<\tfrac14$. Then there exists a constant $\alpha \in (0,\E[D_1]]$ such that if $\lambda>1/\alpha$, then there exists a constant $c>0$ such that 
\[
\E_{[1]_N}(\tau_{[0]_N}) \geq \ee^{cN} \quad \whp.
\]
\end{theorem}

\noindent
For the random regular graph with degree $d \geq 3$, the claim holds with $\alpha = d-2$.   


\subsection{CP on the Erd\H{o}s-R\'enyi random graph}

Let $p=p_N$ be the edge retention probability and $\lambda=\lambda_N$ the infection rate.

\begin{theorem}
\label{thm:4.6}
{\rm (Cator, Don \cite{CD21}.)}\\  
(a) 
If $\lim_{N\to\infty} Np_N = \infty$, then
\[
\E_{[1]_N}(\tau_{[0]_N}) \geq \ee^{c_N N} \quad \whp
\]
for $c_N < \log (Np_N\lambda_N) - \frac{1}{Np_N\lambda_N} + 1$ and $N$ large enough.\\   
(b)
If $\lim_{N\to\infty} Np_N = \sigma \in (4\log 2,\infty)$, then
\[
\E_{[1]_N}(\tau_{[0]_N}) \geq \ee^{cN} \quad \whp
\]
for $\lambda>1/\sigma$ and $c < \log (\sigma\lambda) - \frac{1}{\sigma\lambda} + 1$ and $N$ large enough.
\end{theorem}


\subsection{CP on the preferential attachment model}

\begin{theorem}
\label{thm:4.7}
{\rm (Berger, Borgs, Chayes, Saberi \cite{BBCS14}, Can \cite{C15}).}\\
There exists a $c>0$ such that
\[
\log \E_{[1]_N}(\tau_{[0]_N}) \geq c\,\frac{\lambda^2N}{(\log N)^{\frac{1+\gamma}{1-\gamma}}}
\]
for $\lambda>0$ small enough and $N$ large enough, with $\gamma \in [0,1)$ a parameter that controls the attachment probabilities of newly incoming vertices.
\end{theorem}
  

\subsection{CP on tree-like random graphs}

Many sparse random graphs are locally tree-like, and hence it is interesting to study the extinction time of the CP on regular trees. Let
\[
0 < \lambda_d < \infty
\]
be the critical threshold for survival on the $d$-regular tree.

\begin{theorem}
\label{thm:4.8}
{\rm (Mourrat, Valesin \cite{MV16}).}\\
Let $\mathcal{G}_{d,N}$ be the class of connected graphs with $N$ vertices and maximal degree $d$.Then, for $\lambda < \lambda_d$,
\[
\lim_{N\to\infty} \sup_{G \in \mathcal{G}_{d,N}} \P_{[1]_G}(\tau_{[0]_G} > c \log N) = 0
\quad \text{for some } c = c(\lambda,d).
\]
\end{theorem}

\begin{theorem}
\label{thm:4.9}
{\rm (Mourrat, Valesin \cite{MV16}).}\\ 
On the $d$-regular random graph with $d \geq 3$, the crossover from logarithmic to exponential extinction time occurs at $\lambda_d$.
\end{theorem} 

\begin{theorem}
\label{thm:4.10}
{\rm (Baptista da Silva, Oliveira, Valesin \cite{BdSOV21pr}, Schapira, Valesin \cite{SV23pr}).}\\
On the dynamic $d$-regular random graph with $d \geq 3$ and rewiring rate $\nu>0$, the crossover occurs at a strictly smaller value than $\lambda_d$.
\end{theorem}   

The rewiring helps the infection spread more easily.


\subsection{CP on general finite graphs}

What can be said about the extinction time for the CP on general finite graphs?

\begin{theorem}
\label{thm:4.11}
{\rm (Mountford, Mourrat, Valesin, Yao \cite{MMVY16}).}\\
For any $\lambda>\lambda_1$, any $D\in\N$ and any connected graph $G$ whose degrees are bounded by $D$,
\[ 
\begin{aligned}
&\E_{[1]_G}(\tau_{[0]_G}) \geq \exp\big[c|V|\big]\\[0.3cm]  
&\text{for some} c = c(\lambda,D) > 0.
\end{aligned}
\]
\end{theorem}

\begin{theorem}
\label{thm:4.12}
{\rm (Schapira, Valesin \cite{SV17}).}\\
For any $\lambda>\lambda_1$, any $\epsilon>0$ and any connected graph $G$,
\[
\E_{[1]_G}(\tau_{[0]_G}) \geq \exp\left[\frac{c|V|}{(\log |V|)^{1+\epsilon}}\right]
\quad \text{for some } c = c(\epsilon) > 0.
\]
\end{theorem}

\noindent
CONCLUSION: For the CP it is hard to get sharp control on the extinction time, and many questions remain open.



\end{document}